\documentclass[reqno,a4paper]{article}

\usepackage{a4wide}
\usepackage[centertags]{amsmath}
\usepackage{eucal,dsfont,accents,bbm,amsfonts,amssymb,amsthm,amsopn,amscd,units}
\usepackage[usenames]{color}
\definecolor{citegreen}{rgb}{0,0.6,0}
\definecolor{refred}{rgb}{0.8,0,0}
\usepackage[colorlinks, citecolor=citegreen, linkcolor=refred]{hyperref}

\title{The higher-dimensional Chern-Gauss-Bonnet formula for singular conformally flat manifolds}
\author{Reto Buzano and Huy The Nguyen}
\date{}

\providecommand{\abs}[1]{\lvert #1\rvert}

\newcommand{\Rc}{\mathrm{Rc}}
\newcommand{\R}{\mathrm{R}}
\newcommand{\RR}{\mathbb{R}}
\newcommand{\Sph}{\mathbb{S}}
\newcommand{\eps}{\varepsilon}
\newcommand{\Lap}{\triangle}

\newcommand{\vol}{\mathrm{vol}}
\newcommand{\dB}{\partial B}

\def\Xint#1{\mathchoice
   {\XXint\displaystyle\textstyle{#1}}
   {\XXint\textstyle\scriptstyle{#1}}
   {\XXint\scriptstyle\scriptscriptstyle{#1}}
   {\XXint\scriptscriptstyle\scriptscriptstyle{#1}}
   \!\int}
\def\XXint#1#2#3{{\setbox0=\hbox{$#1{#2#3}{\int}$}
     \vcenter{\hbox{$#2#3$}}\kern-.5\wd0}}

\def\dashint{\Xint-}
\DeclareMathOperator{\dist}{dist}

\theoremstyle{plain}
\newtheorem{lemma}{Lemma}[section]
\newtheorem{prop}[lemma]{Proposition}
\newtheorem{thm}[lemma]{Theorem}
\newtheorem{cor}[lemma]{Corollary}
\newtheorem{rem}[lemma]{Remark}
\newtheorem{defn}[lemma]{Definition}
\newtheorem{claim}{Claim}

\numberwithin{equation}{section}

\begin{document}
\maketitle

\begin{abstract}
\noindent
In a previous article, we generalised the classical four-dimensional Chern-Gauss-Bonnet formula to a class of manifolds with finitely many conformally flat ends and singular points, in particular obtaining the first such formula in a dimension higher than two which allows the underlying manifold to have isolated conical singularities. In the present article, we extend this result to all even dimensions $n\geq 4$ in the case of a class of conformally flat manifolds.
\end{abstract}

\section{Introduction}
Among the most fundamental results in differential geometry is the Gauss-Bonnet theorem which relates the Gauss curvature $K_g$ of a closed and smooth Riemannian surface $(M^2,g)$ with its Euler characteristic $\chi(M)$ by the formula
\begin{equation*}
\chi(M) = \frac{1}{2\pi} \int_M K_g \, dV_g.
\end{equation*}
Dropping the assumption that the surface is closed, the formula generally requires correction terms as can already be seen by looking at the plane (with $K=0$ and $\chi=1$). Moreover, these correction terms certainly do not only depend on the topology of the underlying surface, but rather on the geometry of its ends. This can be observed by considering a flat cylinder and a catenoid, two surfaces that are topologically equivalent (with Euler characteristic $\chi=0$) but satisfy $K\equiv 0$ and $K<0$, respectively. Similarly, also when the smoothness assumption is dropped, the Gauss-Bonnet formula requires correction terms. A simple example which illustrates this is the object obtained by cutting out a slice of a two-sphere between two geodesics from its north to its south pole and gluing the resulting bi-gon back together along these geodesics. The ``sphere with two conical points'' created this way is locally isometric to the round sphere away from the poles (and thus in particular satisfies $K\equiv 1$), but due to the fact that we lost area, the Gauss-Bonnet formula cannot hold. In fact, we see that the formula needs a correction term which depends on the area that was cut out, or equivalently on the conical angle at the two singularities. As it turns out, all correction terms can be expressed as isoperimetric ratios at the ends or isoperimetric deficits at the singular points. For only some of the most important results of the extensive literature on such two-dimensional formulas, we refer the reader to \cite{CV35,H57,F65,T91,CL95}.\\

A higher-dimensional analog of the Gauss-Bonnet formula has been discovered by Chern \cite{Ch43}. In dimension four, it can be expressed as
\begin{equation}\label{eq.CGB4d}
\chi(M) = \frac{1}{4\pi^2}\int_M \Big(\frac{1}{8}\abs{W_g}_g^2+Q_{g,4}\Big) dV_g,
\end{equation}
where $(M^4,g)$ is a smooth closed four-manifold, $W_g$ is its Weyl curvature, and
\begin{equation}\label{eq.Q4d}
Q_{g,4} := -\frac{1}{12}\big(\Lap_g \R_g -\R_g^2 +3\abs{\Rc_g}_g^2\big),
\end{equation}
is the Paneitz $Q$-curvature introduced by Branson and Branson-{\O}rsted \cite{B85,BO91,B93}. Here, $\R_g$ denotes the scalar and $\Rc_g$ the Ricci curvature of $(M^4,g)$. As in the two-dimensional case, also the four-dimensional formula \eqref{eq.CGB4d} requires correction terms if the smoothness or compactness assumptions are dropped. The most basic situation where this can be observed is for a conformal metric on $\RR^4\setminus\{0\}$ with one end (at infinity) and one singular point (at the origin). For such metrics, we proved the following result in \cite{BuzaNg1}. 

\begin{thm}[\cite{BuzaNg1}, Theorem 1.1]\label{thm.R4}
Let $g=e^{2w}\abs{dx}^2$ be a metric on $\RR^4\setminus\{0\}$ which is complete at infinity and has finite area over the origin. If $g$ has finite total $Q$-curvature, $\int_{\RR^4}\abs{Q_{g,4}}\, dV_g<\infty$, and non-negative scalar curvature at infinity and at the origin, then we have
\begin{equation}\label{eq.CGB-R4}
\chi(\RR^4)-\frac{1}{4\pi^2}\int_{\RR^4}Q_{g,4}\, dV_g = \nu-\mu,
\end{equation}
where $\nu := \lim_{r\to\infty} C_{3,4}(r)$, $\mu :=\lim_{r\to 0} C_{3,4}(r) - 1$, and 
\begin{equation*}
C_{3,4}(r) := \frac{\vol_g(\dB_r(0))^{4/3}}{4(2\pi^2)^{1/3}\,\vol_g(B_r(0))}
\end{equation*}
denotes the isoperimetric ratio, normalised to be $1$ on Euclidean space.
\end{thm}

More generally, one can consider four-manifolds $(M^4,g)$ with finitely many conformally flat complete ends and finite-area singular points, that is
\begin{equation}\label{eq.M4split}
M=N\cup\Big(\bigcup_{i=1}^k E_i\Big)\cup\Big(\bigcup_{j=1}^{\ell} S_j\Big)
\end{equation}
where $(N,g)$ is a compact manifold with boundary $\partial N=\big(\bigcup_{i=1}^k \partial E_i\big)\cup\big(\bigcup_{j=1}^{\ell} \partial S_j\big)$, each $E_i$ is a conformally flat complete simple end satisfying
\begin{equation}\label{eq.4dends}
(E_i,g)=(\RR^4\setminus B,e^{2e_i}\abs{dx}^2)
\end{equation}
for some function $e_i(x)$, and each $S_j$ is a conformally flat region with finite area and with a point-singularity at some $p_j$, satisfying
\begin{equation}\label{eq.4dsings}
(S_j\setminus\{p_j\},g)=(B\setminus\{0\},e^{2s_j}\abs{dx}^2)
\end{equation}
for some function $s_j(x)$. Here, $B$ denotes the unit ball in $\RR^4$. Localising Theorem \ref{thm.R4} to such ends and singular regions (obtaining Chern-Gauss-Bonnet formulas with boundary terms), and gluing all the pieces together, we obtained the following more general theorem.

\begin{thm}[\cite{BuzaNg1}, Theorem 1.6]\label{thm.M4}
Let $(M^4,g)$ satisfy \eqref{eq.M4split}--\eqref{eq.4dsings} and assume that $g$ has finite total $Q$-curvature, $\int_{M}\abs{Q_{g,4}}\, dV_g<\infty$, and non-negative scalar curvature at every singular point and at infinity at each end. Then we have
\begin{equation}\label{eq.CGB-M4}
\chi(M)-\frac{1}{4\pi^2}\int_M \Big(\frac{1}{8}\abs{W_g}_g^2+Q_{g,4}\Big) dV_g = \sum_{i=1}^k\nu_i-\sum_{j=1}^{\ell}\mu_j,
\end{equation}
where in the coordinates of \eqref{eq.4dends} and \eqref{eq.4dsings}, we have
\begin{equation*}
\nu_i :=\lim_{r\to\infty}\, \frac{\big(\int_{\dB_r(0)}e^{3e_i(x)}d\sigma(x)\big)^{4/3}}{4(2\pi^2)^{1/3}\int_{B_r(0)\setminus B}e^{4e_i(x)}dx},\qquad i=1,\ldots,k,
\end{equation*}
and
\begin{equation*}
\mu_j :=\lim_{r\to0}\, \frac{\big(\int_{\dB_r(0)}e^{3s_j(x)}d\sigma(x)\big)^{4/3}}{4(2\pi^2)^{1/3}\int_{B_r(0)}e^{4s_j(x)}dx}-1,\qquad j=1,\ldots,\ell.
\end{equation*}
\end{thm}

The Chern-Gauss-Bonnet formulas in the above two theorems, generalising in particular the formulas of Chang-Qing-Yang \cite{CQY1,CQY2} for smooth but non-compact four-manifolds, are the first such formulas in a dimension higher than two which allow the underlying manifold to have isolated branch points or conical singularities. It is natural to ask whether Theorem \ref{thm.R4} and \ref{thm.M4} can be generalised to higher even dimensions $n = 2m \geq 4$ using the $n$-dimensional $Q$-curvature. In the present article, we give an affirmative answer in the case of Theorem \ref{thm.R4} and prove an analog of Theorem \ref{thm.M4} for a class of conformally flat manifolds.\\

Restricting to the conformally flat case has two main advantages. On the one hand, the Paneitz operator (see \cite{P83}) and its corresponding $Q$-curvature are not unique for general manifolds but for a conformally flat metric $g=e^{2w}\abs{dx}^2$, the $n$-dimensional $Q$-curvature is uniquely determined. We define the $n$-dimensional $Q$-curvature by the formula
\begin{equation}\label{eq.Qn}
2Q_{g,n} := e^{-nw(x)} (-\Lap)^{n/2} w(x).
\end{equation} 
As we only work in even dimensions, this is indeed an $n$-th order partial differential equation (while it would yield a pseudo-differential equation involving a fractional Laplacian in odd dimensions). Let us remark here that an explicit formula for $Q_{g,n}$ in terms of the Riemann curvature tensor and its covariant derivatives, similar to \eqref{eq.Q4d} in the four-dimensional case, is difficult to obtain in higher dimensions and is currently unknown for dimensions higher than $8$. On the other hand, a second advantage of restricting to conformally flat manifolds is that in this case the $Q$-curvature is a multiple of the Pfaffian modulo a divergence term. Thus the Chern-Gauss-Bonnet theorem can be written as
\begin{equation}\label{eq.CGBn}
\chi(M^n) = \frac{1}{\gamma_n} \int_{M^n} Q_{g,n}\, dV_g.
\end{equation}
for a smooth and closed (locally) conformally flat $n$-manifold $(M^n,g)$ with 
\begin{equation}\label{eq.gamman}
\gamma_n = 2^{n-2} (\tfrac{n-2}{2})! \, \pi^{n/2},
\end{equation}
We first prove the following generalisation of our four-dimensional result from Theorem \ref{thm.R4}.

\begin{thm}\label{thm.Rn}
Let $n\geq 4$ be an even integer and let $g=e^{2w}\abs{dx}^2$ be a metric on $\RR^n\setminus\{0\}$ which is complete at infinity and has finite area over the origin. If $g$ has finite total $Q$-curvature, $\int_{\RR^n}\abs{Q_{g,n}}\, dV_g<\infty$, and non-negative scalar curvature at infinity and at the origin, then we have
\begin{equation}\label{eq.CGB-Rn}
\chi(\RR^n)-\frac{1}{\gamma_n}\int_{\RR^n}Q_{g,n}\, dV_g = \nu-\mu,
\end{equation}
where 
\begin{equation*}
\nu := \lim_{r\to\infty} \frac{\vol_g(\dB_r(0))^{n/(n-1)}}{n\, \sigma_n^{1/(n-1)}\,\vol_g(B_r(0))} \qquad \mu := \lim_{r\to 0} \frac{\vol_g(\dB_r(0))^{n/(n-1)}}{n\, \sigma_n^{1/(n-1)}\,\vol_g(B_r(0))} - 1.
\end{equation*}
Here, $\sigma_n = \abs{\Sph^{n-1}} = 2\pi^{n/2} / (\tfrac{n-2}{2})!$ denotes the surface area of the unit $(n-1)$-sphere.
\end{thm} 

\begin{rem}\label{counterexample}
By non-negative scalar curvature at infinity we mean
\begin{equation}\label{eq.nonnegRcorrect}
\inf_{\RR^n\setminus B_{r_1}(0)} \R_g(x)\geq 0
\end{equation}
for some $0<{r_1}<\infty$. This assumption cannot be replaced by the weaker assumption 
\begin{equation}\label{eq.nonnegRfalse}
\liminf_{\abs{x}\to\infty} \R_g(x)\geq 0
\end{equation}
as can be seen by considering the metric $g=e^{2\abs{x}^2}\abs{dx}^2$, which is complete at infinity, satisfies $Q_{g,n}\equiv 0$ and 
\begin{equation*}
\R_g(x)=-n(n-1)e^{-2\abs{x}^2}-(n-1)(n-2)\abs{x}^2e^{-2\abs{x}^2}\to 0 \quad (\text{as }\abs{x}\to\infty),
\end{equation*}
but has asymptotic isoperimetric ratio $\nu = +\infty$ (and $\mu=0$, since the metric is smooth at the origin). Similarly, we also require $\inf_{B_{r_2}(0)} \R_g(x)\geq 0$ for some $0<{r_2}<\infty$.
\end{rem}

Using a partition of unity argument, we then obtain the following more general theorem.
\begin{thm}\label{thm.Mn}
Let $\Lambda = \{p_1, \ldots, p _k, q _1 , \ldots q_\ell \} \subset \mathbb{S}^n$ be a finite (possibly empty) set of points and let $(\Omega = \mathbb{S}^n \setminus \{p_1, \ldots, p _k\},g)$ be such that each $p_ i$ is a complete end of finite total $Q$-curvature and non-negative scalar curvature and each $q _ j$ is a finite area singular point of finite total $Q$-curvature and non-negative scalar curvature. Then 
\begin{align*}
\chi(\Omega) - \frac{1}{\gamma_n}\int_{\Omega}Q_{g,n} \, dV_g = \sum _{i=1}^ k \nu_i - \sum _{ j=1} ^\ell \mu_j ,
\end{align*}  
where 
\begin{equation*}
\nu_i := \lim_{r_{p_i}\to 0} \frac{\vol_g(\dB_{r_{p_1}}(p_i))^{n/(n-1)}}{n\, \sigma_n^{1/(n-1)}\,\vol_g(B_R(p_i) \setminus B_{r_{p_1}}(p_i))}, \quad \mu_j := \lim_{r_{q_j }\to 0} \frac{\vol_g(\dB_{r_{q_j }}(q_j))^{n/(n-1)}}{n\, \sigma_n^{1/(n-1)}\,\vol_g(B_{r_{q_j }}(q_j))} - 1
\end{equation*}
with $\R_{p}(x)= \dist(x , p)$, and $R$ is chosen small enough such that the balls $\{B_R(p)\}_{p\in\Lambda}$ are pairwise disjoint. 
\end{thm}

To put our results into context, let us compare them with what was previously known. In \cite{F05}, Fang proved the Chern-Gauss-Bonnet inequality 
\begin{equation}\label{eq.Fang}
\chi(M^n)-\frac{1}{\gamma_n}\int_{M^n}Q_{g,n}\, dV_g \geq 0
\end{equation}
for \emph{smooth} locally conformally flat $n$-manifolds with finitely many ends where the scalar curvature is non-negative. Our Theorems \ref{thm.Rn} and \ref{thm.Mn} do not only generalise this result to manifolds with singularities, but we also get an explicit formula for the error terms in this inequality. Note that for singular manifolds the inequality \eqref{eq.Fang} might no longer hold if the conical angle (and hence the isoperimetric ratio) is larger than the Euclidean one. Later, Ndiaye and Xiao \cite{NX10} obtained a result similar to Theorem \ref{thm.Rn}, but again only for \emph{smooth} metrics. They did not consider the case of manifolds with several ends but only studied conformal metrics on $\RR^n$. Let us remark that their work also contains some small errors which we correct here, in particular in their main theorem they only assume \eqref{eq.nonnegRfalse} and not \eqref{eq.nonnegRcorrect}, which in view of the counterexample from the above remark is not a sufficiently strong assumption. Nevertheless, we clearly profited from their results and, in fact, some of the ideas in the present article are in parts inspired by \cite{F05,NX10} -- combined with the approach we developed in \cite{BuzaNg1} in order to deal with isolated singularities. Let us also mention here that various normalisations of the $Q$-curvature exist in the literature and that, in particular, \cite{F05,NX10} define $Q_{g,n}$ without the factor $2$ in \eqref{eq.Qn}. We prefer to put this factor to make the results consistent with formula \eqref{eq.Q4d} and with our earlier work \cite{BuzaNg1} in dimension four.\\

In the \emph{singular} case, the problem has previously only been studied for manifolds with edge-cone singularities and $V$-manifolds, see e.g.~\cite{LS98,AL13,S57}, but these results are technically very different and none of them allows for isolated singular points.\\

Let us now describe how this article is organised and how the arguments differ from our previous four-dimensional results in \cite{BuzaNg1}. First, in Section \ref{sect.est}, we collect some integral estimates which are needed for the arguments in the following sections. Then, in the Sections \ref{sect.symm}--\ref{sect.general}, we prove Theorems \ref{thm.Rn} in three steps as follows. In Section \ref{sect.symm}, we prove the result in the special case where $w=w(r)$ is a radial function on $\RR^n\setminus\{0\}$ and \eqref{eq.Qn} reduces to an ODE. To make it easier to deal with all dimensions at once, we do not solve explicitly for the non-linearity in this ODE as we did in \cite{BuzaNg1} but rather prove asymptotic estimates for an abstract kernel (see Lemma \ref{lemma.fa}), employing the integral estimates from Section \ref{sect.est}. In Section \ref{sect.normal}, we introduce an $n$-dimensional version of our notion of generalised normal metrics from \cite{BuzaNg1} and prove Theorem \ref{thm.Rn} for this class of metrics. Then, in Section \ref{sect.general}, we show that every metric $g=e^{2w}\abs{dx}^2$ on $\RR^n\setminus\{0\}$ satisfying the assumptions of Theorem \ref{thm.Rn} is actually such a generalised normal metric. The necessary singularity removal argument in this section is much more involved than our short argument in dimension $n=4$ from \cite{BuzaNg1}, as we cannot directly apply B\^{o}cher's Theorem in higher dimensions. Finally, in Section \ref{sect.gluing}, we prove Theorem \ref{thm.Mn}.  As currently no good notion of boundary $T$-curvature associated to $Q_{g,n}$ is known in dimensions $n>4$, we cannot localise Theorem \ref{thm.Rn} as we did in \cite{BuzaNg1}. However, in the special case of domains $\Omega$ as in Theorem \ref{thm.Mn}, we can deduce the desired result from an easy partition of unity argument.


\section{Some Integral Estimates}\label{sect.est}

Let $n\geq 4$ be an even integer. Recall (see e.g.~\cite[Thm.~6.20]{LL01}) that $G_y(x) = \frac{1}{\abs{x-y}^{n-2}}$ is a multiple of the fundamental solution of the Laplacian on $\RR^n$. Indeed, in the sense of distributions, we have
\begin{equation}\label{eq.fundLap}
(-\Lap_x)G_y(x) = (n-2)\sigma_n \cdot \delta_y(x),
\end{equation}
where $\sigma_n = \abs{\Sph^{n-1}}$. This has two immediate consequences that are useful for us. First, it shows that in the sense of distributions
\begin{equation}\label{eq.fundLapn2}
(-\Lap_x)^{n/2} \log \frac{1}{\abs{x-y}} = 2\gamma_n \cdot \delta_y(x),
\end{equation}
where $\gamma_n$ is given in \eqref{eq.gamman}. Second, we see that if $u$ is a radial function solving $(-\Lap)^k u=0$ on $\RR^n\setminus \{0\}$, where $k$ is an integer satisfying $1\leq k\leq \frac{n}{2}$, then $u$ is given by
\begin{equation}\label{eq.soluformula}
\begin{aligned}
u(x) &= c_1 + c_2\frac{1}{\abs{x}^{n-2}} + c_3\abs{x}^2 + c_4\frac{1}{\abs{x}^{n-4}} + \ldots\\
&\quad + c_{n-3}\abs{x}^{n-4} + c_{n-2}\frac{1}{\abs{x}^2} + c_{n-1}\abs{x}^{n-2} + c_n\log\abs{x},
\end{aligned}
\end{equation}
for some constants $c_i\in\RR$, where $c_\ell=0$ for all $\ell>2k$. Both of these conclusions follow form \eqref{eq.fundLap} with short calculations, noting that in spherical coordinates $x=r \sigma$, where $r>0$ and $\sigma\in\Sph^{n-1}$, we have
\begin{equation}\label{eq.Lapspherical}
\Lap = \frac{d^2}{dr^2} + \Big(\frac{n-1}{r}\Big) \frac{d}{dr} + \frac{\Lap_{\Sph^{n-1}}}{r^2}.
\end{equation}

The goal of this section is to derive some integral estimates which will be useful in the arguments that follow. We denote by
\begin{equation*}
\dashint_{\dB_r(0)} f(x) \, d\sigma(x) := \frac{1}{\abs{\dB_r(0)}}\int_{\dB_r(0)} f(x) \, d\sigma(x)
\end{equation*}
the averaged integral over an $(n-1)$-sphere of radius $r$ in Euclidean $\RR^n$ and consider the following four integrals:
\begin{align}
I_n(r,s) &:= \dashint_{\dB_r(0)} \frac{1}{\abs{x-y}^{n-2}} \, d\sigma(x),\label{eq.defI}\\
J_n(r,s) &:= \dashint_{\dB_r(0)} \frac{1}{\abs{x-y}^{2}} \, d\sigma(x),\label{eq.defJ}\\
K_n(r,s) &:= \dashint_{\dB_r(0)} \frac{\abs{\abs{x}^2-\abs{y}^2}}{\abs{x-y}^{2}} \, d\sigma(x),\label{eq.defK}\\
L_n(r,s) &:= \dashint_{\dB_r(0)} \log \frac{\abs{y}}{\abs{x -y}} \, d\sigma(x),\label{eq.defL}
\end{align}
where $r = \abs{x}$, $s=\abs{y}$.

\begin{prop}\label{prop.intest}
We have the following estimates for the integrals from \eqref{eq.defI}--\eqref{eq.defL}.
\begin{itemize}
\item[i)] For any $r,s > 0$ and even integer $n\geq 4$, $I_n(r,s)$ evaluates to
\begin{equation}\label{eq.estI}
I_n(r,s) =
\begin{cases}
\frac{1}{r^{n-2}}, & \text{ if $s\leq r$},\\ 
\frac{1}{s^{n-2}}, & \text{ if $s> r $}.
\end{cases}
\end{equation}
\item[ii)] There exists $C>0$ such that for any $r,s > 0$ and even integer $n\geq 4$, we have
\begin{equation}\label{eq.prelimestJ}
\begin{aligned}
\abs{r^2 J_n(r,s)-1} &\leq C \frac{s^2}{r^2}, \quad  \text{ if $s \leq r$},\\
J_n(r,s) &\leq C\frac{1}{s^2}, \quad \text{ if $s>r$}.
\end{aligned}
\end{equation}
In particular, in both cases, we obtain 
\begin{equation}\label{eq.estJ}
r^2 J_n(r,s) \leq C.
\end{equation}
\item[iii)] There exists $C>0$ such that for any $r,s > 0$ and even integer $n\geq 4$, we have
\begin{equation}\label{eq.estK}
K_n(r,s)\leq C.
\end{equation}
\item[iv)] There exists $C>0$ such that for any $r,s > 0$ satisfying $\frac{1}{2}r\leq s\leq \frac{3}{2}r$ and any even integer $n\geq 4$, we have
\begin{equation}\label{eq.estL}
\abs{L_n(r,s)} \leq C.
\end{equation}
\end{itemize}
\end{prop}

\begin{proof} 
i) This follows analogously to our corresponding estimate in the four-dimensional case \cite[(3.17)]{BuzaNg1}, but in order to make the article self-contained, we quickly give a proof here. From \eqref{eq.fundLap}, we obtain
\begin{equation}\label{eq.rderivativeI1}
\begin{aligned}
\int_{\dB_r(0)} \partial_r \frac{1} {\abs{x -y}^{n-2}}\,d\sigma(x) &= \int_{B_r(0)} \Lap_ x \frac{1}{\abs{x-y}^{n-2}}\, d\sigma(x)\\
&=-\int_{B_r(0)} (n-2)\sigma_n\delta_y(x) \, d\sigma(x).
\end{aligned}
\end{equation}
Therefore, if $s \leq r$, we have $y \in B_r(0)$ and hence obtain 
\begin{equation*}
\partial _ r \, \dashint_{\dB_r(0)} \frac{1}{\abs{x-y}^{n-2}}\, d \sigma(x) = \dashint_{\dB_r(0)} \partial_r \frac{1} {\abs{x -y}^{n-2}}\,d\sigma(x)= -\frac{(n-2)\sigma_n}{\abs{\dB_r(0)}} = -\frac{n-2}{r^{n-1}}. 
\end{equation*}
Integrating with respect to $r$, we see that there is a constants $C_1(s)$ such that
\begin{equation*}
I_n(r,s) = \dashint_{\dB_r (0)} \frac{1}{\abs{x-y}^{n-2}}\, d\sigma(x) = \frac{1}{r^{n-2}} + C_1(s).
\end{equation*}
Fixing $y$ (and hence $\abs{y} = s$) and letting $r\to \infty$, we see that $C_1(s) = 0$. This proves \eqref{eq.estI} in the first case. If instead we have $s>r$, \eqref{eq.rderivativeI1} implies
\begin{equation*}
\partial_r \, \dashint_{\dB_r(0)} \frac{1}{\abs{x-y}^{n-2}}\, d\sigma(x) = 0,
\end{equation*}
so that 
\begin{equation*}
I_n(r,s) = \dashint_{\dB_r(0)} \frac{1}{\abs{x-y}^{n-2}} \, d\sigma(x) = C_2(s).
\end{equation*}
As $s > r$, we can fix $s=\abs{y}$ and let $r \to 0$ this time to obtain 
\begin{equation*}
\lim_{r\to 0} I_n(r,s) = \dashint_{\dB_r (0)} \frac{1}{\abs{y}^{n-2}}\, d\sigma(x) = \frac{1}{s^{n-2}} = C_2(s).
\end{equation*}
This proves \eqref{eq.estI} in the second case.\\

ii) If $n=4$, we have $J_n(r,s) = I_n(r,s)$, and thus \eqref{eq.estI} yields the desired result. We can therefore assume that $n$ is an even integer satisfying $n\geq 6$. In the case $s\leq r$, we compute
\begin{equation*}
\Lap_y^{\frac{n}{2} - 2} J_n(r,s) = \dashint_{\dB_r(0)} \Lap_y^{\frac{n}{2} - 2} \frac{1}{\abs{x-y}^2} \, d\sigma(x) = \dashint_{\dB_r(0)}  \frac{C}{\abs{x-y}^{n-2}} \, d\sigma(x) = C\cdot I_n(r,s). 
\end{equation*} 
We therefore obtain
\begin{equation*}
\Lap_y^{\frac{n}{2} - 2} J_n(r,s) =\frac{C}{r^{n-2}},  \qquad  \Lap_y^{\frac{n}{2}-1} J_n(r,s) =0,
\end{equation*}
and thus, by \eqref{eq.soluformula}, we have
\begin{align*}
J_n(r,s) = \sum_{k=0}^{\frac{n}{2}-2} c_{2k+1}(r) s^{2k}  +\sum_{k=1}^{\frac{n}{2}-1} c_{n-2k}(r) s^{-2k}. 
\end{align*}
As $\lim_{s \to 0} J_n(r,s) = \frac{1}{r^2}$, we see that $c_1(r)= \frac{1}{r^2}$ and that all the coefficients in the second sum above have to vanish. This means that
\begin{equation*}
J_n(r,s) = \frac{1}{r^2} + \sum_{k=1}^{\frac{n}{2}-2} c_{2k+1}(r) s^{2k}.
\end{equation*}
Note that by definition, for any $t>0$,
\begin{align*}
(tr)^2  J_n(tr,ts) = r^2 J_n(r,s),
\end{align*}
so that $r^2 J_n(r,s) = 1 + p \big(\frac{s^2} {r^2} \big)$, where $p$ is a polynomial of degree $\frac{n}{2}-2$ with no constant term. Obviously, this is bounded when $s \leq r$, proving the first case of \eqref{eq.prelimestJ}. If $s > r$, we simply apply H\"older's inequality with $p=\frac{n}{2}-1$, yielding
\begin{align*}
J_n(r,s) &= \dashint_{ \dB_r(0)} \frac{1}{\abs{x -y}^2} \,d\sigma(x) \leq \bigg( \dashint_{\dB_r(0)} \frac{1}{\abs{x-y}^{2p}} \,d\sigma(x) \bigg)^{1/p}\\
& = \big(I_n(r,s)\big)^{2/(n-2)} = \Big(\frac{1}{s^{n-2}}\Big)^{2/(n-2)} =\frac{1}{s^2}. 
\end{align*} 
This proves the second case of \eqref{eq.prelimestJ}.\\

iii) For given $\abs{y}=s$, let $A_s:=\{x\in\RR^n \mid \abs{x}<\frac{1}{2}\abs{y} \text{ or } \abs{x}>2\abs{y}\}$ and notice, by simply plugging in both possible cases, that on $A_s$ we have
\begin{equation*}
\frac{\abs{x}+\abs{y}}{\abs{\abs{x}-\abs{y}}} \leq 3.
\end{equation*}
Hence by the triangle inequality $\abs{\abs{x}-\abs{y}}\leq \abs{x-y}$, we find
\begin{equation}\label{eq.KAs}
\frac{\abs{\abs{x}^2-\abs{y}^2}}{\abs{x-y}^2} \leq \frac{\abs{\abs{x}^2-\abs{y}^2}}{\abs{\abs{x}-\abs{y}}^2} = \frac{\abs{x}+\abs{y}}{\abs{\abs{x}-\abs{y}}} \leq 3 \qquad \text{on } A_s.
\end{equation}
On the complement of $A_s$, we have $\frac{1}{2}\abs{x} \leq \abs{y} \leq 2\abs{x}$ and therefore obtain
\begin{equation}\label{eq.KAscomp}
\frac{\abs{\abs{x}^2-\abs{y}^2}}{\abs{x-y}^2} \leq \frac{3\abs{x}^2}{\abs{x-y}^2} \qquad \text{on } \RR^n \setminus A_s.
\end{equation}
Using \eqref{eq.KAs} and \eqref{eq.KAscomp}, we can now estimate
\begin{align*}
K_n(r,s) &= \dashint_{\dB_r(0)} \frac{\abs{\abs{x}^2-\abs{y}^2}}{\abs{x-y}^2} \, d\sigma(x) \leq \dashint_{\dB_r(0)} 3\Big(1 + \frac{\abs{x}^2}{\abs{x-y}^2} \Big) d\sigma(x)\\
&= 3\bigg(1+r^2\dashint_{\dB_r(0)} \frac{1}{\abs{x-y}^2} \, d\sigma(x)\bigg) = 3(1 + r^2 J_n(r,s)).
\end{align*}
In view of \eqref{eq.estJ}, the bound \eqref{eq.estK} follows.\\

iv) Again, we already proved this (in dimension four) in \cite{BuzaNg1}, but quickly repeat the short proof here. First, we estimate the integral by
\begin{align*}
\abs{L_n(r,s)} &=\left| \dashint_{\dB_r(0)} \log \frac{\abs{y}}{\abs{x -y}} d\sigma(x) \right| \\
&\leq \frac{1}{\abs{\dB_r(0)}} \int_{\dB_r(0)\setminus \{x\, :\, \abs{x-y} < \frac{1}{3} \abs{y}\}} \left| \log \frac{\abs{y}}{\abs{x-y}} \right| d\sigma(x)\\
&\quad +\frac{1}{\abs{\dB_r(0)}} \int_{\dB_r(0)\cap \{x\, :\, \abs{x-y} < \frac{1}{3} \abs{y}\}} \left| \log \frac{\abs{y}}{\abs{x-y}} \right| d\sigma(x)\\
& =: L_n^1(r,s) + L_n^2(r,s)
\end{align*}
and then estimate the two terms $L_n^1(r,s)$ and $L_n^2(r,s)$ separately. In order to estimate the first term, note that combining the assumption $\frac{1}{2}\abs{x} \leq \abs{y} \leq \frac{3}{2}\abs{x}$ with $\abs{x-y} \geq \frac{1}{3} \abs{y}$ implies
\begin{equation*}
\frac{1}{3} \abs{y}\leq\abs{x-y}\leq\abs{x}+\abs{y}\leq 3\abs{y},
\end{equation*}
which in turn shows that $\frac{1}{3}\leq \frac{\abs{y}}{\abs{x-y}}\leq 3$ over the region where we integrate. Hence, we obtain $L_n^1(r,s) \leq \log 3$. For the second integral, we have
\begin{align*}
L_n^2(r,s) \leq \frac{1}{\big| \dB_{\frac{r}{\abs{y}}}(0)\big|} \int_{\dB_{\frac{r}{\abs{y}}}(0)\cap \left\{x\, :\, \left |x -\frac{y}{\abs{y}} \right| < \frac{1}{3}\right\}} \left| \log \frac{1}{\abs{x-\frac{y}{\abs{y}}}} \right| d \sigma(x).
\end{align*} 
Now, the assumption $\frac{1}{2}r\leq \abs{y} \leq \frac{3}{2}r$ gives $\frac{2}{3} \leq \frac {r}{\abs{y}} \leq 2$. Thus, as $ \log \frac{1}{\left| x-\frac{y}{\abs{y}} \right|}$ is integrable, we have that $L_n^2(r,s)$ is uniformly bounded on the annulus $\frac{2}{3} \leq \frac {r}{\abs{y}} \leq 2$. The uniform bounds for $L_n^1(r,s)$ and $L_n^2(r,s)$ imply \eqref{eq.estJ}.
\end{proof}

\section{The Rotationally Symmetric Case}\label{sect.symm}

In this section, we prove Theorem \ref{thm.Rn} for a conformal metric $g = e^{2w}\abs{dx}^2$ on $\RR^n\setminus\{0\}$ when $w=w(r)$ is a \emph{radial} function. Here and in the following, we always use the notation $r=\abs{x}$. In this situation, \eqref{eq.Qn} becomes an ODE, but instead of solving for the nonlinearity in this ODE as we did in our previous four-dimensional work \cite{BuzaNg1} by an explicit integration, we rather apply the integral estimates derived in the previous section. We first prove the following lemma.

\begin{lemma} \label{lemma.fa}
Let $n\geq 4$ be an even integer and let $w \in C^\infty (\RR^n\setminus \{0\})$ be rotationally symmetric such that $g = e^{2w}\abs{dx}^2$ has finite total $Q$-curvature
\begin{equation}\label{eq.finitetotQ}
\int_{\RR^n} \abs{Q_{g,n}} dV_g = \int_{\RR^n} \abs{Q_{g,n}} e^{nw} dx < \infty.
\end{equation}
For $\alpha\in\RR$, define
\begin{equation*}
f_\alpha(x) := \frac {1}{\gamma_n} \int_{\RR^n} \log\Big(\frac{\abs{y}}{\abs{x-y}}\Big)\, Q_{g,n}(y)\,e^{nw(y)}\,dy+\alpha\log\abs{x}.
\end{equation*}
Then $f_\alpha(x) = f_\alpha(r)$ is radially symmetric and satisfies 
\begin{equation}\label{eq.limitdifference}
\lim_{r \to \infty } r\frac {df_\alpha}{dr}(r) - \lim_{r \to 0}  r\frac {df_\alpha}{dr}(r) = - \frac{1}{\gamma_n} \int_{\RR^n} Q_{g,n}\, dV_g.
\end{equation}
Moreover, there exists a constant $C\in\RR$ such that
\begin{equation}\label{eq.asympf}
\abs{x}\abs{\nabla f_\alpha} < C, \qquad \abs{x}^2\abs{\Lap f_\alpha} < C.
\end{equation}
\end{lemma}

\begin{proof}
The radial symmetry of $f_\alpha$ follows easily from the fact that $w$ and $\log |x|$ are both radially symmetric. As in Section 3 of \cite{BuzaNg1}, we have
\begin{equation}\label{eq.ddrlog}
\frac {d}{dr} \log \frac{\abs{y}}{\abs{x-y}} = -\left \langle \frac{\nabla_x \abs{x-y}}{\abs{x-y}},\frac{x}{\abs{x}} \right\rangle = -\left \langle \frac{x-y}{\abs{x-y}^2},\frac{x}{\abs{x}} \right\rangle= - \frac{\abs{x-y}^2+\abs{x}^2-\abs{y}^2}{2\abs{x}\abs{x-y}^2}.
\end{equation}
Hence, using the radial symmetry of $w$, we obtain
\begin{align*}
r \frac{df_\alpha}{dr}(r)  &= - \frac{1}{2\gamma_n} \int_{\RR^n} \bigg(\frac{\abs{x-y}^2+\abs{x}^2-\abs{y}^2}{\abs{x-y}^2}\bigg) \, Q_{g,n}(y) \, dV_g(y) + \alpha\\
& = - \frac{1}{2\gamma_n} \int_{\RR^n} \bigg(1 + \dashint_{\dB_{\abs{x}}(0)} \frac{\abs{z}^2-\abs{y}^2}{\abs{z-y}^2} \, d\sigma(z) \bigg) Q_{g,n}(y) \, dV_g(y) + \alpha.
\end{align*}
Now, we get that 
\begin{align*}
\bigg| \int_{\RR^n} \bigg(1 + \dashint_{\dB_{\abs{x}}(0)} &\frac{\abs{z}^2-\abs{y}^2}{\abs{z-y}^2} \, d\sigma(z) \bigg) Q_{g,n}(y) \, dV_g(y) \bigg| \\
&\leq C \int_{\RR^n} \Big(1 + \sup_{0<r,s<\infty} K_n(r,s) \Big) \abs{Q_{g,n}(y)}\, dV_g(y)\\
&\leq C \int_{\RR^n} \abs{Q_{g,n}(y)} \, dV_g(y) < \infty,
\end{align*}
where $K_n(r,s)$ is defined in \eqref{eq.defK} and we used Proposition \ref{prop.intest} iii) as well as the assumption of finite total $Q$-curvature \eqref{eq.finitetotQ} in the last steps. We can therefore apply the dominated convergence theorem, implying
\begin{equation*}
\lim_{r\to 0 } r \frac {df_\alpha}{dr}(r) = - \frac{1}{2\gamma_n} \int_{\RR^n} \lim_{\abs{x}\to 0} \bigg( \frac{\abs{x-y}^2+\abs{x}^2-\abs{y}^2}{\abs{x-y}^2} \bigg) Q_{g,n}(y) \, dV_g(y) + \alpha = \alpha.
\end{equation*}
Similarly, as for any fixed $y$ we have
\begin{equation*}
\lim_{\abs{x}\to \infty}  \bigg( \frac{\abs{x-y}^2+\abs{x}^2-\abs{y}^2}{\abs{x-y}^2} \bigg) =2,
\end{equation*}
we get, again by the dominated convergence theorem,
\begin{align*}
\lim_{r\to \infty } r \frac {df_\alpha}{dr}(r) &= - \frac{1}{2\gamma_n} \int_{\RR^n} \lim_{\abs{x}\to \infty} \bigg( \frac{\abs{x-y}^2+\abs{x}^2-\abs{y}^2}{\abs{x-y}^2} \bigg) Q_{g,n}(y) \, dV_g(y) + \alpha\\
&= - \frac{1}{\gamma_n} \int_{\RR^n} Q_{g,n}(y) \, dV_g(y) + \alpha\\
&= - \frac{1}{\gamma_n} \int_{\RR^n} Q_{g,n}(y) \, dV_g(y) + \lim_{r\to 0 } r \frac {df_\alpha}{dr}(r),
\end{align*}
which proves \eqref{eq.limitdifference}. In order to prove \eqref{eq.asympf}, we work in spherical coordinates $x=r \sigma$, where $r>0$ and $\sigma\in\Sph^{n-1}$. In these coordinates, we have $\nabla = (\frac{d}{dr}, \frac{1}{r} \nabla_{\Sph^{n-1}})$. Therefore, as $f_\alpha(x)$ is rotationally symmetric, the above yields
\begin{equation*}
\limsup_{\abs{x}\to 0} \, \abs{x}\abs{\nabla f_\alpha(x)} < \infty \qquad\text{and}\qquad \limsup_{\abs{x}\to \infty} \, \abs{x}\abs{\nabla f_\alpha(x)} < \infty,
\end{equation*}
and hence the first claim in \eqref{eq.asympf}. To prove the second claim, we note that, using \eqref{eq.Lapspherical}, we find
\begin{equation}\label{eq.Laplog}
\Lap \log r = \frac {n-2} {r^2},
\end{equation}
and therefore
\begin{align*}
\Lap f_\alpha(x) &= - \frac{n-2}{\gamma_n} \int_{\RR^n}  \frac {1}{\abs{x-y}^2} \,  Q_{g,n}(y) \, dV_g(y) + \alpha\cdot \frac{n-2}{r^2}\\
& = - \frac{n-2}{\gamma_n} \int_{\RR^n} \left(\dashint_{\dB_{\abs{x}}(0)} \frac {1}{\abs{z-y}^2} \, d\sigma(z) \right) Q_{g,n}(y) \, dV_g(y) + \alpha\cdot \frac{n-2}{r^2},
\end{align*}
where -- as above -- we used the rotational symmetry of $w$ and hence of $Q_{g,n}$. Using $J_n(r,s)$ defined in \eqref{eq.defJ} and estimated in \eqref{eq.estJ}, we then easily get
\begin{align*}
\abs{\Lap f_\alpha(x)} &\leq C \bigg| \int_{\RR^n} \left(\dashint_{\dB_{\abs{x}}(0)} \frac {1}{\abs{z-y}^2} \, d\sigma(z) \right) Q_{g,n}(y) \, dV_g(y) \bigg| + \frac{C}{r^2}\\
&\leq C \int_{\RR^n}  \sup_{0<r,s<\infty} J_n(r,s) \cdot \abs{Q_{g,n}(y)} \, dV_g(y) + \frac{C}{r^2}\\
&\leq \frac{C}{r^2} \int_{\RR^n} \abs{Q_{g,n}(y)} \, dV_g(y) + \frac{C}{r^2} = \frac{C}{r^2}.
\end{align*}
Therefore $\abs{x}^2\abs{\Lap f_\alpha(x)}$ is uniformly bounded, thus finishing the proof of the lemma.
\end{proof}

Next, we prove that $w$ agrees with some $f_\alpha$ up to a constant.

\begin{lemma} \label{lemma.wisfa}
Let $w$ be as in Lemma \ref{lemma.fa} and assume in addition either that
\begin{equation}\label{eq.posscalar}
\inf_{\RR^n\setminus B_{r_1}(0)} \R_g(x) \geq 0, \qquad \inf_{B_{r_2}(0)} \R_g(x) \geq 0,
\end{equation}
for some $0<r_2\leq r_1 < \infty$, or alternatively that for some constant $C\in\RR$
\begin{equation}\label{eq.asympw}
\abs{x}\abs{\nabla w} < C, \qquad \abs{x}^2\abs{\Lap w} < C.
\end{equation}
Then there exist $\alpha$ and $C\in \RR$ such that 
\begin{equation*}
w(x)= f_\alpha(x)+ C =\frac {1}{\gamma_n} \int_{\RR^n} \log\Big(\frac{\abs{y}}{\abs{x-y}}\Big)\, Q_{g,n}(y)\,e^{nw(y)}\,dy+\alpha\log\abs{x} + C.
\end{equation*}
\end{lemma}

\begin{proof}
We first define the function
\begin{equation*}
f(x) := f_0(x) =\frac {1}{\gamma_n} \int_{\RR^n} \log\Big(\frac{\abs{y}}{\abs{x-y}}\Big)\, Q_{g,n}(y)\,e^{nw(y)}\,dy.
\end{equation*}
From \eqref{eq.fundLapn2}, we see that 
\begin{equation*}
(-\Lap)^{n/2}f(x) = \frac {1}{\gamma_n} \int_{\RR^n} \Big( (-\Lap_x)^{n/2}\log \frac{1}{\abs{x-y}}\Big) Q_{g,n}(y)\,e^{nw(y)}\,dy = 2Q_{g,n}(x)\,e^{nw(x)}.
\end{equation*}
Thus, from the definition of $Q$-curvature in \eqref{eq.Qn}, we obtain $(-\Lap)^{n/2}(w-f) = 0$. As $(w-f)$ is a radial function, we can use \eqref{eq.soluformula} to conclude that
\begin{equation*}
(w-f)(x) = c_1 + c_n\log\abs{x} + \sum_{k=1}^{\frac{n}{2}-1} c_{2k+1}\abs{x}^{2k} + \sum_{k=1}^{\frac{n}{2}-1} c_{n-2k}\abs{x}^{-2k}.
\end{equation*}
Setting $C:=c_1$ and $\alpha:= c_n$, we can rewrite this as
\begin{equation}\label{eq.wfsums}
w(x) = f_\alpha(x) + C + \sum_{k=1}^{\frac{n}{2}-1} c_{2k+1}\abs{x}^{2k} + \sum_{k=1}^{\frac{n}{2}-1} c_{n-2k}\abs{x}^{-2k}.
\end{equation}
We need to show that all the coefficients $c_\ell$ in the above sums vanish. This is done similarly as in the four-dimensional case, see \cite[Lemma 2.2]{BuzaNg1}, using either the scalar curvature assumption \eqref{eq.posscalar} or the bounds \eqref{eq.asympw}. In fact, the formula for the scalar curvature of a conformally flat metric $g=e^{2w}\abs{dx}^2$ is
\begin{equation}\label{eq.Rconf}
\R_g e^{2w} = - \tfrac{4(n-1)}{n-2} e^{-\frac{n-2}{2}w} \Lap \big(e^{\frac{n-2}{2}w}\big) = - 2(n-1)\big( \Lap w + \big(\tfrac{n}{2}-1\big)\abs{\nabla w}^2\big),
\end{equation}
and therefore
\begin{equation}\label{eq.formulaR}
\abs{x}^2 \R_g e^{2w} = - 2(n-1)\big( \abs{x}^2 \Lap w + \big(\tfrac{n}{2}-1\big) (\abs{x}\abs{\nabla w})^2\big).
\end{equation}
By the scalar curvature assumption \eqref{eq.posscalar}, the quantity in \eqref{eq.formulaR} is non-negative for $\abs{x}\geq r_1$ as well as for $\abs{x}\leq r_2$. Under the alternative assumption \eqref{eq.asympw}, this expression does not have a sign, but it is uniformly bounded. We now use the bounds for $f_\alpha(x)$ from \eqref{eq.asympf} to conclude that if any of the odd-index coefficients $c_{2k+1}$ in \eqref{eq.wfsums} does not vanish, then the right hand side of \eqref{eq.formulaR} will tend to $-\infty$ as $\abs{x}\to\infty$, giving a contradiction. Similarly, if any of the even-index coefficients $c_{n-2k}$ in \eqref{eq.wfsums} does not vanish, then the left hand side of \eqref{eq.formulaR} tends to $-\infty$ as $\abs{x}\to 0$, yielding again a contradiction. Therefore $w(x)=f_\alpha(x)+C$.
\end{proof}

\begin{rem}
The proof of Lemma \ref{lemma.wisfa} shows that the scalar curvature assumption \eqref{eq.posscalar} can be replaced by the (much less geometric) condition that $\abs{x}^2 \R_g e^{2w}$ is bounded from below. Note that the metric $g=e^{2\abs{x}^2}\abs{dx}^2$ from the counter-example stated in Remark \ref{counterexample} satisfies $\R_g\to 0$ as $\abs{x}\to\infty$, but has $\abs{x}^2 \R_g e^{2w} \to -\infty$ as $\abs{x}\to\infty$. Hence, (the first part of) condition \eqref{eq.posscalar} cannot be weakened to $\liminf_{\abs{x}\to \infty} \R_g(x) \geq 0$ as claimed in \cite{NX10}. A similar argument shows that also the second part of condition \eqref{eq.posscalar} cannot be weakened to $\liminf_{\abs{x}\to 0} \R_g(x) \geq 0$.
\end{rem}

Combining Lemma \ref{lemma.wisfa} with Equation \eqref{eq.limitdifference}, we immediately obtain the following consequence, which corresponds to \cite[Corollary 2.3]{BuzaNg1} in the four-dimensional case.
\begin{cor}\label{cor.CGBw}
Let $g=e^{2w}\abs{dx}^2$ be as in Lemma \ref{lemma.wisfa}. Then
\begin{equation*}
\chi(\RR^n) - \frac{1}{\gamma_n}\int_{\RR^n} Q_{g,n} \, dV_g = \nu - \mu,
\end{equation*}
where
\begin{equation*}
\nu :=\lim_{r\to\infty}r\frac{dw}{dr}(r)+1,\qquad \mu :=\lim_{r\to 0}r\frac{dw}{dr}(r).
\end{equation*}
\end{cor}

To finish the proof of Theorem \ref{thm.Rn} in the rotationally symmetric case, it remains to express the two limits in Corollary \ref{cor.CGBw} as isoperimetric ratios. Let us first recall the mixed volumes $V_k(\Omega)$ defined by Trudinger \cite{T97} for a convex domain $\Omega\subseteq \RR^n$. We can restrict to the special situation where $\Omega$ is a ball $B_r(0)$ and $k=n$ or $k=n-1$, in which case we obtain the following.

\begin{defn}\label{defn.ratios}
We define the $k$-volumes $V_k(r)$ for $k=n$ or $k=n-1$ by
\begin{equation}\label{eq.Vs1}
\begin{aligned}
V_n(r) &=\int_{B_r(0)} e^{nw}\, dx, \\
V_{n-1}(r) &=\frac{1}{n}\int_{\dB_r(0)} e^{(n-1)w}\, d\sigma(x).
\end{aligned}
\end{equation}
Moreover, we define the isoperimetric ratio $C_{n-1,n}(r)$ by
\begin{equation}\label{eq.defiso}
C_{n-1,n}(r) = \frac{V_{n-1}(r)^{n/(n-1)}}{\omega_n^{1/(n-1)} V_n(r)} = \frac{\vol_g(\dB_r(0))^{n/(n-1)}}{n\, \sigma_n^{1/(n-1)}\,\vol_g(B_r(0))},
\end{equation}
where $\omega_n = \sigma_n/n = \abs{B_1(0)}$ is the Euclidean volume of a unit ball in $\RR^n$. Note that the normalisation in \eqref{eq.defiso} is such that $C_{n-1,n}(r) \equiv 1$ on Euclidean space.
\end{defn}

Theorem \ref{thm.Rn} for rotationally symmetric metrics then follows from Corollary \ref{cor.CGBw} combined with the following result.

\begin{lemma}\label{lemma.isoperimetric}
Let $w$ be as in Lemma \ref{lemma.wisfa} and assume in addition that $g$ is complete at infinity and has finite area over the origin. Then we have
\begin{align*}
\nu &=\lim_{r\to\infty}r\frac{dw}{dr}(r) + 1 = \lim_{r\to\infty} C_{n-1,n}(r),\\
\mu &= \lim_{r\to 0}r\frac{dw}{dr}(r) =\lim_{r\to 0} C_{n-1,n}(r) -1.
\end{align*}
\end{lemma}

\begin{proof}
Due to the rotational symmetry of $w(x)$, we can rewrite the $k$-volumes defined in \eqref{eq.Vs1} as
\begin{equation}\label{eq.Vs2}
\begin{aligned}
V_n(r) &= \int_0^r \int_{\dB_s(0)} e^{nw(s)} d\sigma ds = \sigma_n \int_0^r s^{n-1}e^{nw(s)} ds,\\
V_{n-1}(r) &= \frac{1}{n}\int_{\dB_r(0)} e^{(n-1)w(r)}\, d\sigma = \frac{1}{n}\sigma_n r^{n-1} e^{(n-1)w(r)},
\end{aligned}
\end{equation}
yielding
\begin{equation}\label{eq.Vderivatives}
\begin{aligned}
\frac{d}{dr} V_n(r) &=\sigma_n r^{n-1}e^{nw(r)},\\
\frac{d}{dr} V_{n-1}(r) &= \frac{n-1}{n}\sigma_n r^{n-2} e^{(n-1)w(r)} \Big[ r \frac{dw}{dr}(r) +1\Big].
\end{aligned}
\end{equation}
Due to the assumption that $g=e^{2w}\abs{dx}^2$ has finite area over the origin, we obtain that $V_n(r)\to 0$ and $V_{n-1}(r)\to 0$ as $r\to 0$ and therefore by L'H\^{o}pital's rule
\begin{equation*}
\lim_{r\to 0} C_{n-1,n}(r) = \lim_{r\to 0} \frac{\frac{n}{n-1}V_{n-1}(r)^\frac{1}{n-1}\cdot \frac{d}{dr}V_{n-1}(r)}{\omega_n^{1/(n-1)} \frac{d}{dr}V_n(r)} = \lim_{r\to 0} r\frac{dw}{dr}(r) +1.
\end{equation*}
For $r\to\infty$, there are several cases we have to consider.\\
 
Since $V_{n-1}(r)$ is monotone, we have $r\frac{dw}{dr}(r) +1 \geq 0$ for all $r$ by \eqref{eq.Vderivatives}. In the case where $\lim_{r\to \infty} r\frac{dw}{dr}(r) + 1> 0$, \eqref{eq.Vs2}--\eqref{eq.Vderivatives} imply that both $V_n(r)$ and $V_{n-1}(r)$ tend to infinity and therefore we can again apply L'H\^{o}pital's rule as above, obtaining
\begin{equation*}
\lim_{r\to \infty} C_{n-1,n}(r) = \lim_{r\to \infty} r\frac{dw}{dr}(r) +1.
\end{equation*}

We hence assume that $\lim_{r\to \infty} r\frac{dw}{dr}(r) +1= 0$ and consider two subcases: If $V_n(r)$ tends to infinity, then the claim is true (either again by L'H\^{o}pital's rule if also $V_{n-1}(r)$ tends to infinity or otherwise we trivially get $\lim_{r\to\infty} C_{n-1,n}(r)=0$). If $V_n(r)$ is bounded as $r\to \infty$, then \eqref{eq.Vs2} shows that $e^{nw(r)} \to 0$ as $r\to \infty$, which in turn implies that $e^{(n-1)w(r)} \to 0$ and hence again by \eqref{eq.Vs2} $V_{n-1}(r)\to 0$ as $r\to\infty$. Therefore, also in this case, we trivially find $\lim_{r\to\infty} C_{n-1,n}(r)=0 = \lim_{r\to \infty} r\frac{dw}{dr}(r)+1$.\\

We have checked all possible cases and thus finished the proof of Lemma \ref{lemma.isoperimetric} and in view of Corollary \ref{cor.CGBw} also proved Theorem \ref{thm.Rn} for rotationally symmetric metrics.
\end{proof}

In the remainder of this section, we study the case where the origin is a second complete end rather than a finite area singular point. This result will be used in Section \ref{sect.gluing}.

\begin{lemma} \label{lemma.2ends}
Let $n\geq 4$ be an even integer and let $w \in C^\infty (\RR^n\setminus \{0\})$ be rotationally symmetric such that $g = e^{2w}\abs{dx}^2$ has two complete ends (at infinity and at the origin) satisfying \eqref{eq.posscalar} or \eqref{eq.asympw} and with finite total $Q$-curvature $\int_{\RR^n} \abs{Q_{g,n}} dV_g < \infty$.
Then
\begin{equation*}
- \frac{1}{\gamma_n}\int_{\RR^n} Q_{g,n} \, dV_g = \nu_1 + \nu_2,
\end{equation*}
where
\begin{equation*}
\nu_1 := \lim_{r\to \infty} \frac{\vol_g(\dB_r(0))^{n/(n-1)}}{n\, \sigma_n^{1/(n-1)}\,\vol_g(B_r(0)\setminus B_R(0))}, \quad \nu_2 := \lim_{r\to 0} \frac{\vol_g(\dB_r(0))^{n/(n-1)}}{n\, \sigma_n^{1/(n-1)}\,\vol_g(B_R(0)\setminus B_r(0))}
\end{equation*}
for an arbitrary $R>0$.
\end{lemma}

This lemma in particular has the consequence that no such metric $g$ can have positive $Q$-curvature everywhere. The proof is almost identical to the above, and we therefore only give a short sketch.

\begin{proof}
By Corollary \ref{cor.CGBw}, subtracting $\chi(\RR^n)=1$ on both sides, we have
\begin{equation*}
- \frac{1}{\gamma_n}\int_{\RR^n} Q_{g,n} \, dV_g = \nu - (\mu+1),
\end{equation*}
where
\begin{equation*}
\nu :=\lim_{r\to\infty}r\frac{dw}{dr}(r)+1,\qquad \mu+1 :=\lim_{r\to 0}r\frac{dw}{dr}(r) +1.
\end{equation*}
Following the proof of Lemma \ref{lemma.isoperimetric} for the complete end at infinity, we obtain $\nu = \nu_1$. For the end at the origin, we define
\begin{equation*}
\tilde{V}_n(r) :=\int_{B_R(0)\setminus B_r(0)} e^{nw}\, dx = \int_r^R \int_{\dB_s(0)} e^{nw(s)} d\sigma ds = \sigma_n \int_r^R s^{n-1}e^{nw(s)} ds
\end{equation*}
which satisfies
\begin{equation*}
\frac{d}{dr} \tilde{V}_n(r) = - \sigma_n r^{n-1}e^{nw(r)}.
\end{equation*}
Therefore, denoting
\begin{equation*}
\tilde{C}_{n-1,n}(r) = \frac{V_{n-1}(r)^{n/(n-1)}}{\omega_n^{1/(n-1)} \tilde{V}_n(r)} = \frac{\vol_g(\dB_r(0))^{n/(n-1)}}{n\, \sigma_n^{1/(n-1)}\,\vol_g(B_R(0)\setminus B_r(0))},
\end{equation*}
a computation using L'H\^{o}pital's rule as in Lemma \ref{lemma.isoperimetric} implies that
\begin{equation*}
\nu_2 = \lim_{r\to 0} \tilde{C}_{n-1,n}(r) = -\Big(\lim_{r\to 0} r\frac{dw}{dr}(r) +1\Big) = -(\mu + 1).
\end{equation*}
This finishes the proof of Lemma \ref{lemma.2ends}.
\end{proof}


\section{Generalised Normal Metrics}\label{sect.normal}
In this section, we first define generalised normal metrics on $\RR^n$ as an extension of our definition of generalised normal metrics in $\RR^4$ in \cite{BuzaNg1}. This in turn was a generalisation of normal metrics in \cite{CQY1} and \cite{F65}. We then prove Theorem \ref{thm.Rn} for this class of metrics.

\begin{defn}[Generalised normal metrics]\label{defn.normal}
Suppose that $g=e^{2w}\abs{dx}^2$ is a metric on $\RR^n \setminus \{0\} $ with finite total Paneitz $Q$-curvature 
\begin{equation*}
\int_{\RR^n} \abs{Q_{g,n}} e^{nw} dx < \infty. 
\end{equation*}   
We call $g$ a \emph{generalised normal metric}, if $w$ has the expansion
\begin{equation}\label{eq.defnormal}
w(x)=\frac{1}{\gamma_n} \int_{\RR^n}\log\Big(\frac{\abs{y}}{\abs{x-y}}\Big)\, Q_{g,n}(y)\,e^{nw(y)}\,dy+\alpha\log\abs{x}+C
\end{equation}
for some constants $\alpha,C\in\RR$. For such a generalised normal metric, we then define the \emph{averaged metric} $\bar{g}=e^{2\bar{w}}\abs{dx}^2$ by
\begin{equation}
\bar{w}(r):=\dashint_{\dB_r(0)}w(x)\,d\sigma(x) = \frac{1}{\abs{\dB_r(0)}}\int_{\dB_r(0)}w(x)\,d\sigma(x).
\end{equation}
Clearly, $\bar{g}$ is a rotationally symmetric metric.
\end{defn}

\begin{prop} \label{prop.1}
Suppose that the metric $g= e ^{2w} \abs{dx}^2 $ on $ \RR^n \setminus\{0 \}$ is a generalised normal metric. Then for all $k > 0$, we have that 
\begin{align}\label{eq.lemma1}
\dashint _{ \dB_r(0) } e^ {kw(x)} d \sigma(x) = e ^{ k \bar w ( r ) } e ^{o(1)}
\end{align}  
where $o(1)\to 0$ as $ r\to \infty $ or $ r \to 0$.
\end{prop}
\begin{proof}
The proof of this statement is merely a modification of our four-dimensional version from \cite[Lemma 3.4]{BuzaNg1}. The proof for $r\to\infty$ was essentially covered in \cite[Lemma 3.2]{CQY1} in the four-dimensional case and later generalised in \cite[Prop. 3.1 (ii)]{NX10} to higher dimensions. Note that in \cite{CQY1,NX10} the formula \eqref{eq.lemma1} is proved for \emph{normal metrics} which differ from our definition of \emph{generalised} normal metrics by our additional term $\alpha\log\abs{x}$ in \eqref{eq.defnormal}. But this additional term, the fundamental solution of the $\frac{n}{2}$-Laplacian, is rotationally symmetric and thus in equation \eqref{eq.lemma1}, $e^{\alpha \log\abs{x}}$ appears on both sides and hence cancels. For this reason, we only need to prove the proposition for $r\to 0$.\\

Suppose that $g=e^{2w}\abs{dx}^2$ is a generalised normal metric. To simplify notation, we denote $F(y)=Q_g(y)e^{nw(y)}$, which by assumption is in $L^1$. Then, splitting up $\RR^n$ into three regions, we have
\begin{align*}
w(x) &= \frac{1}{\gamma_n} \int_{B_{\abs{x}/2}(0)} \log \left( \frac{\abs{x}}{\abs{x-y}}\right)F(y) dy + f(\abs{x})\\
&\quad + \frac{1}{\gamma_n} \int_{\RR^n \setminus B_{3\abs{x}/2}(0)} \log \left( \frac{\abs{y}}{\abs{x-y}}\right)F(y) dy\\
&\quad +\frac{1}{\gamma_n} \int_{B_{3\abs{x}/2}(0)\setminus B_{\abs{x}/2}(0)} \log \left( \frac{\abs{y}}{\abs{x-y}} \right)F(y) dy + \alpha \log\abs{x} + C \\
& = w_1(x) + w_2(x) + w_3(x) + f(\abs{x}) + \alpha \log\abs{x} + C,
\end{align*}
where
\begin{equation*}
f(\abs{x}) =\frac{1}{\gamma_n} \int_{B_{\abs{x}/2}(0)} \log \left( \frac{\abs{y}}{\abs{x}}\right)F(y) dy.
\end{equation*}
We note that $f(\abs{x}) +\alpha\log\abs{x}+C$ is rotationally symmetric and hence in equation \eqref{eq.lemma1}, $\exp(f(\abs{x})+\alpha \log\abs{x}+C)$ appears on both sides and thus cancels. Therefore, we need to study only $w_1(x)$, $w_2(x)$ and $w_3(x)$. We first claim the following.
\begin{claim}\label{claim1}
We have
\begin{equation}\label{eq.proofstep1}
w_1(x) = o(1), \qquad \text{as $\abs{x}\to 0$}
\end{equation} 
and
\begin{equation}\label{eq.proofstep2}
w_2(x) = o(1), \qquad \text{as $\abs{x}\to 0$}.
\end{equation} 
\end{claim}
\begin{proof}
In order to prove \eqref{eq.proofstep1}, let $\eta<\frac{1}{2}$, and estimate 
\begin{equation}\label{eq.absw1}
\abs{w_1(x)} \leq C \left( \int_{\abs{y}\leq \eta\abs{x}} \left| \log \frac{\abs{x}}{\abs{x-y}}\right| \abs{F(y)} dy + \int_{\eta\abs{x}\leq\abs{y}\leq \frac{1}{2} \abs{x}} \left| \log \frac {\abs{x}}{\abs{x-y}}\right| \abs{F(y)} dy\right).
\end{equation}
Note that $\abs{y}\leq\eta\abs{x}$ implies
\begin{equation*}
(1-\eta)\abs{x}\leq \abs{x}-\abs{y}\leq\abs{x-y}\leq\abs{x}+\abs{y}\leq (1+\eta)\abs{x},
\end{equation*}
and therefore
\begin{equation*}
\left| \log\frac{\abs{x}}{\abs{x-y}}\right| \leq \max \left\{ \left| \log \frac{1}{1+\eta} \right|, \left| \log \frac{1}{1-\eta}\right| \right\} = \log \frac{1}{1-\eta}.
\end{equation*}
We use this to estimate the first integral in \eqref{eq.absw1}. For the second integral, we then use the bound $\abs{y}\leq\frac{1}{2}\abs{x}$, which by an analogous argument as above yields
\begin{equation*}
\left| \log\frac{\abs{x}}{\abs{x-y}}\right| \leq \left |\log \frac{1}{1-\frac{1}{2}} \right| =\log 2.
\end{equation*}
Combining the two estimates and using $\int_{\RR^n}\abs{F(y)}dy < \infty$, we obtain
\begin{align*}
\abs{w_1(x)} \leq C \log \frac{1}{1-\eta} + \log 2 \int_{\eta \abs{x} \leq \abs{y} \leq \frac{1}{2}\abs{x}} \abs{F(y)}dy.
\end{align*}
For $\abs{x}\to 0$ and $\eta\to 0$, both terms above tend to zero, using again $\int_{\RR^n}\abs{F(y)}dy < \infty$. This proves \eqref{eq.proofstep1}.\\

The argument to prove \eqref{eq.proofstep2} is dual to what we have just done. For $\eta>\frac{3}{2}$, we write 
\begin{equation}\label{eq.absw2}
\abs{w_2(x)} \leq C \left( \int_{\abs{y}\geq \eta\abs{x}} \left| \log \frac{\abs{y}}{\abs{x-y}}\right| \abs{F(y)} dy + \int_{\eta\abs{x}\geq\abs{y}\geq \frac{3}{2} \abs{x}} \left| \log \frac {\abs{y}}{\abs{x-y}}\right| \abs{F(y)} dy\right).
\end{equation}
Then, we note that $\abs{y}\geq \eta\abs{x}$ yields
\begin{equation*}
(1-\tfrac{1}{\eta})\abs{y}\leq \abs{y}-\abs{x}\leq\abs{x-y}\leq\abs{y}+\abs{x}\leq (1+\tfrac{1}{\eta})\abs{y},
\end{equation*}
which gives
\begin{equation*}
\left| \log\frac{\abs{y}}{\abs{x-y}}\right| \leq \max \left\{ \left| \log \frac{1}{1+\frac{1}{\eta}} \right|, \left| \log \frac{1}{1-\frac{1}{\eta}}\right| \right\} = \log \frac{1}{1-\frac{1}{\eta}}.
\end{equation*}
This can be used to estimate the first integral in \eqref{eq.absw2}. Similarly, we estimate the second integral, using the bound $\abs{y}\geq \frac{3}{2}\abs{x}$, which gives
\begin{equation*}
\left| \log\frac{\abs{y}}{\abs{x-y}}\right| \leq \log \frac{1}{1-\frac{2}{3}} = \log 3.
\end{equation*}
Combining these estimates and using $\int_{\RR^n}\abs{F(y)}dy < \infty$, we have
\begin{align*}
\abs{w_2(x)} \leq C \log \frac{\eta}{\eta-1} + \log 3 \int_{\eta \abs{x} \geq \abs{y} \geq \frac{3}{2}\abs{x}} \abs{F(y)}dy.
\end{align*}
For $\abs{x}\to 0$, we can send $\eta\to\infty$ slow enough such that $\eta\abs{x}\to 0$, in which case both terms above tend to zero, using again $\int_{\RR^n}\abs{F(y)}dy < \infty$. This yields \eqref{eq.proofstep2} and thus finishes the proof of Claim \ref{claim1}.
\end{proof}

\begin{claim}\label{claim2}
We have
\begin{equation}\label{eq.proofstep3a}
\dashint _{ \dB_r(0)} w_3 (x) d \sigma(x) = o(1), \quad \text{ as $ r\to 0$},
\end{equation}
as well as
\begin{equation}\label{eq.proofstep4}
\left| \dashint_{\dB_r(0)} (e^{kw_3(x)} -1) d\sigma(x) \right| = o(1), \quad \text{ as $ r\to 0$, for all $k>0$}.
\end{equation}
\end{claim}
\begin{proof}
By Fubini's theorem 
\begin{align*}
\bar{w}_3(r)&=\dashint_{\dB_r(0)} w_3(x) d\sigma(x)\\
&= \frac{1}{\gamma_n} \int_{\frac{r}{2}\leq \abs{y} \leq \frac{3r}{2}} \left( \dashint_{\dB_r(0)} \log \frac{\abs{y}}{\abs{x -y}} d\sigma(x) \right)F(y) dy\\
&=\frac{1}{\gamma_n} \int_{\frac{r}{2}\leq  \abs{y} \leq \frac{3r}{2}} L_n(r,s) F(y) dy,
\end{align*}
where $L_n(r,s)$ is as in \eqref{eq.defL}. By \eqref{eq.estL}, we know that $L_n(r,s)$ is uniformly bounded over the region where we integrate. Therefore, using the assumption of finite total $Q$-curvature and the dominated convergence theorem, \eqref{eq.proofstep3a} follows.\\

To prove \eqref{eq.proofstep4}, we follow the idea of Finn \cite{F65} and estimate 
\begin{equation*}
E_M = \{ \sigma \in \Sph^{n-1} : \abs{w_3(r \sigma)} > M \}.
\end{equation*} 
Similar to the above, we have 
\begin{align*}
M \cdot \abs{E_M} &\leq \int_{E_M} \abs{w_3} \, d\sigma\\
&\leq \frac{1}{\gamma_n} \int_{B_{3r/2} (0)\setminus B_{r/2} (0)}
\left( \int _{ E_M} \left| \log \frac {\abs{y}} {\abs{r \sigma - y}} \right| d\sigma \right) \abs{F(y)}d y  \\
 &= \frac{1}{\gamma_n } \int_{\frac {r}{2} \leq \abs{y} \leq \frac {3r}{2}} \tilde{L}_n(r,s) \abs{F(y)} dy,
\end{align*} 
where
\begin{align*}
\tilde{L}_n(r,s) &= \int_{E_M \setminus \left \{\sigma\, :\, \abs{r \sigma - y} \leq \frac{\abs{y}}{3} \right\}} \left| \log \frac {\abs{y}}{\abs{r \sigma - y}}  \right| d\sigma + \int _{ E_M \cap\left \{\sigma\, :\, \abs{r \sigma - y} \leq \frac{\abs{y}}{3} \right \}}  \left| \log \frac {\abs{y}}{\abs{r \sigma - y}}  \right| d\sigma\\
 & =\tilde{L}_n^1(r,s) + \tilde{L}_n^2(r,s). 
\end{align*}
Similar to the estimate of $L_n(r,s)$ in Proposition \ref{prop.intest}, we clearly have $\tilde{L}_n^1(r,s) \leq \log 3 \cdot \abs{E_M}$. We estimate the term $\tilde{L}_n^2(r,s)$ as follows. Observe that if we have 
$\abs{r \sigma - y} \leq \frac {\abs{y}}{ 3}$ then 
\begin{align*}
\log \bigg| \frac {\abs{y}}{\abs{r \sigma -y}} \bigg| \leq  \bigg| \log \frac{\abs{y}}{r} \bigg| + \bigg| \log \Big| \sigma - \frac{y}{r} \Big| \bigg| \leq  \log \frac 3 2 + \bigg| \log \Big| \sigma - \frac{y}{r} \Big| \bigg|.
\end{align*} 
We can thus bound $\tilde{L}_n^2(r,s)$ by the situation where $E_M$ is a $n$-dimensional disc centred at the point $\frac{y}{r}$ orthogonal to $y$, in which case we get
\begin{align*}
\tilde{L}_n^2(r,s) \leq C \abs{E_M} + C \abs{E_M} \log \frac {1}{\abs{E_M}} \leq C \left ( 1 + \log \frac {1}{\abs{E_M}} \right)\abs{E_M}.
\end{align*}
Combining these estimates, we have
\begin{align*}
 M \leq o(1) \left ( 1 + \log \frac{1}{\abs{E_M}} \right),
\end{align*}
where $ o(1) \to 0$ as $ r \to 0$. This implies
\begin{equation*}
\abs{E_M} \leq C e^ {-M / o(1)},
\end{equation*}
and thus
\begin{equation*}
\left| \dashint_{\dB_r(0)} (e^{kw_3(x)} -1) d\sigma(x) \right| = \frac {k}{\abs{\dB_ 1(0)}} \int_{-\infty}^{+\infty} (e^{kM} -1) \abs{E_M} dM = o(1),
\end{equation*}
which finishes the proof of Claim \ref{claim2}.
\end{proof}

Using the two claims, it is easy to finish the proof of Proposition \ref{prop.1}. Combining \eqref{eq.proofstep1}, \eqref{eq.proofstep2} and \eqref{eq.proofstep3a}, as well as Jensen's inequality, we obtain for $r\to 0$
\begin{equation}\label{eq.proofstep3b}
\begin{aligned}
k\, \dashint _{\dB_r (0)} w(x)d \sigma(x) &= \dashint _{\dB_r (0)} k \big(w(x) - w_3(x) \big) d \sigma(x)+ o(1)\\
&= \log \left( \dashint _{\dB_r(0)}e^{k(w(x) - w_3(x))} d\sigma(x)\right) + o(1)
\end{aligned}
\end{equation}  
as $r \to 0$. Combining this with \eqref{eq.proofstep4}, we find 
\begin{align*}
k\bar{w}(r)= k\; \dashint_{\dB_r(0)} w(x) d \sigma(x) = \log \left(\dashint_{\dB_ r(0)}   e^{kw(x)} d \sigma(x)\right) + o(1),
\end{align*}
which is equivalent to \eqref{eq.lemma1}, thus finishing the proof of Proposition \ref{prop.1}
\end{proof}

The Proposition \ref{prop.1} then immediately implies the following two corollaries.

\begin{cor}\label{cor.volumes}
Let $ g$ be a generalised normal metric on $ \RR ^ n \setminus \{ 0 \} $ with averaged metric $ \bar g$ and define the mixed volumes $ V_k$  (with respect to $ g$) and $ \bar V _k $ (with respect to $\bar g$) as in Definition \ref{defn.ratios}. Then we have
\begin{align}
\frac { d }{ dr } V _ n (r ) &= \frac{ d }{ dr} \bar V_n(r) ( 1 + \eps ( r) ), \label{eq.thm3.V4}\\ 
V_{ n-1 } ( r) & = \bar V_{n-1} ( r) ( 1 + \eps (r) ),\label{eq.thm3.V3} 
\end{align}
where $\eps  (r) \to 0$ as $ r \to 0$ or $ r \to \infty$.
\end{cor}

\begin{cor}\label{cor.finitearea}
Suppose that the metric $g = e^{ 2w }\abs{d x}^2$ on $ \RR^n \setminus \{0\} $ is a generalised normal metric that is complete and has a finite area singularity at $ \{ 0 \} $ then $\bar{g}=e^{2 \bar{w}} \abs{dx}^2$ is complete and has a finite area singularity over $ \{ 0 \} $.
\end{cor}

In order to conclude that Theorem \ref{thm.Rn} holds for $\bar w$, we need to either show the geometric property that $\bar g$ has positive scalar curvature at infinity and the origin, as in \eqref{eq.posscalar}, or alternatively verify the analytical assumption \eqref{eq.asympw}. These latter bounds are easy to verify for the averaged conformal factor of a generalised normal metric. In fact, we prove a slightly more general result here which we can then also use in the next section.

\begin{lemma}\label{lemma.growth}
Suppose that $g=e^{2w}\abs{dx}^2$ is a metric on $\RR^n \setminus \{0\} $ with finite total $Q$-curvature 
\begin{equation*}
\int_{\RR^n} \abs{Q_{g,n}} e^{nw} dx < \infty. 
\end{equation*}   
Define $v(x)$ by
\begin{equation}\label{eq.defvw}
v(x)=\frac{1}{\gamma_n} \int_{\RR^n}\log\Big(\frac{\abs{y}}{\abs{x-y}}\Big)\, Q_{g,n}(y)\,e^{nw(y)}\,dy+\alpha\log\abs{x}+C
\end{equation}
for some constants $\alpha,C\in\RR$ and set
\begin{equation}\label{eq.defbarv}
\bar{v}(r):=\dashint_{\dB_r(0)}v(x)\,d\sigma(x).
\end{equation}
Then $\bar v$ satisfies the following bounds 
\begin{equation}\label{eq.asympv}
\abs{\Lap \bar v(r)} \leq \frac {C}{r^2}, \qquad \abs{\nabla \bar{v}(r)} \leq \frac{C}{r},
\end{equation}
for some constant $C \in \RR$.
\end{lemma}   
\begin{proof}
This follows very similarly to the proof of \eqref{eq.asympf}. From the definition of $ \bar{v}(r)$ we get 
\begin{align*}
\Lap \bar v(r) = \dashint_{\dB_r(0)} \Lap v(x) \, d\sigma(x).
\end{align*}
Using \eqref{eq.Laplog} and \eqref{eq.defvw}, we obtain 
\begin{align*}
\Lap \bar v(r) &= \frac{1}{\gamma_n} \int_{\RR^n} \bigg( \, \dashint_{\dB_r(0)} \frac{2-n}{\abs{x - y}^2} \, d\sigma(x) \bigg) Q_{g,n}(y) e^{nw(y)} dy + \alpha\cdot \frac{n-2}{r^2}\\
&= \frac{1}{\gamma_n} \int_{\RR^n} (2-n) \, J_n(r,s) \, Q_{g,n}(y) \, e^{nw(y)} dy + \alpha\cdot \frac{n-2}{r^2}.
\end{align*}
Then from \eqref{eq.estJ}, we see that 
\begin{equation*}
\abs{\Lap \bar v(r)} \leq \frac{C}{r^2} \bigg( \, \int_{\RR^n}\abs{Q_{g,n}(y)}e^{nw(y)} dy + 1 \bigg) \leq \frac{C}{r^2},
\end{equation*}
as we have finite total $Q$-curvature. Similarly, using \eqref{eq.ddrlog}, we deduce that
\begin{align*}
r \frac{d}{dr} \bar v(r) &= r\, \dashint_{\dB_r(0)} \frac{d}{dr}v(x) d\sigma(x)\\
&= -\frac{1}{2\gamma_n} \int_{\RR^n} \bigg(1+ \dashint_{\dB_r(0)} \frac{\abs{x}^2-\abs{y}^2}{\abs{x - y}^2} \, d\sigma(x) \bigg) Q_{g,n}(y) e^{nw(y)} dy + \alpha\\
&= -\frac{1}{2\gamma_n} \int_{\RR^n} \big(1+K_n(r,s)\big) Q_{g,n}(y) \, e^{nw(y)} dy + \alpha.
\end{align*}
Hence, Equation \eqref{eq.estK}, the assumption of finite total $Q$-curvature, and the rotational symmetry of $\bar{v}(r)$ imply
\begin{equation*}
\abs{\nabla \bar v(r)} = \bigg\lvert \frac{d}{dr}\bar v(r)\bigg\rvert \leq \frac{C}{r} \bigg( \, \int_{\RR^n}\abs{Q_{g,n}(y)}e^{nw(y)} dy + 1 \bigg) \leq \frac{C}{r}.
\end{equation*}
This establishes the lemma.
\end{proof}

Obviously, if $g =e^{2w}\abs{dx}^2$ is a generalised normal metric on $ \RR^n\setminus\{0\}$, then $v=w$ and hence \eqref{eq.asympv} gives the desired bounds for $\bar w$. This allows us to now prove Theorem \ref{thm.Rn} under the assumption that $g=e^{2w}\abs{dx}^2$ is a generalised normal metric.

\begin{proof}[Proof of Theorem \ref{thm.Rn} for generalised normal metrics]
Let $g$ be a generalised normal metric with average metric $\bar g$ (see Definition \ref{defn.normal}). Corollary \ref{cor.finitearea}, Lemma \ref{lemma.growth} and the results from the last section show that Theorem \ref{thm.Rn} holds for the rotationally symmetric metric $\bar g$. Moreover, Corollary \ref{cor.volumes} implies that
\begin{align*}
\lim_{r\to 0} C_{n-1,n}(r) &= \lim_{r\to 0} \bar{C}_{n-1,n}(r),\\
\lim_{r\to \infty} C_{n-1,n}(r) &= \lim_{r\to \infty} \bar{C}_{n-1,n}(r),
\end{align*}
where the isoperimetric ratios (given in Definition \ref{defn.ratios}) are taken with respect to $g$ and $\bar g$, respectively. Thus, in order to obtain Theorem \ref{thm.Rn} for the generalised normal metric $g$, we need to only show that
\begin{equation*}
\int_{\RR^n} Q_{g,n} \, dV_g = \int_{\RR^n} Q_{g,n} \, e^{nw} \, dx = \int_{\RR^n} Q_{\bar g,n} \, e^{n\bar w} dx = \int_{\RR^n} Q_{\bar g,n} \, dV_{\bar g}.
\end{equation*}
However, from
\begin{equation*}
2 Q_{\bar{g},n} e^{n \bar w} = (-\Lap)^{n/2} \bar w = \dashint_{\dB_r(0)} (-\Lap)^{n/2} w \, d\sigma =\dashint_{\dB_r(0)} 2Q_{g,n} \, e^{nw} \, d\sigma,
\end{equation*}
this follows immediately.
\end{proof}


\section{Singularity Removal Theorem}\label{sect.general}

Let $g = e ^{2w} \abs{dx}^2$ be a metric on $ \RR^n \setminus \{0\} $ satisfying the assumptions of Theorem \ref{thm.Rn}. In this section, we show that then $g$ is a generalised normal metric. Together with the results from Section \ref{sect.normal}, this completes the proof of Theorem \ref{thm.Rn}.

\begin{prop} \label{prop_gen_metric}
Suppose that the metric $g = e^{2w} \abs{dx}^2$ is a complete finite area metric on $\RR^n \setminus \{0\}$ with finite total $Q$-curvature
\begin{align*}
\int _{ \RR ^ n } \abs{Q_{g,n}} \, e ^{ n w } \, dx < \infty 
\end{align*} 
and non-negative scalar curvature at infinity and at the origin. Then it is a generalised normal metric in the sense of Definition \ref{defn.normal} 
\end{prop}

For $g = e^{2w} \abs{dx}^2$, we define the symmetrisation of $w$ with respect to $ x _0$ by 
\begin{align*}
\bar w _{ x _0 } ( x ) = \dashint _{ \dB _r (x_0) }   w (y) \, d \sigma (y), \quad \text{ where } r = \abs{x - x _0}.
\end{align*}
Clearly $ \bar w_{ x_0 } $ is rotationally symmetric with respect to $x_0$. If $ x_0 = 0$ then we will often write $ \bar w _{ 0}( x ) = \bar w (x)$. 

\begin{lemma}\label{lemma.scalarx0}
Let $g = e^{2w} \abs{dx}^2$ be a complete finite area metric on $\RR^n \setminus \{0\}$ with finite total $Q$-curvature and non-negative scalar curvature at infinity and at the origin. Then for $x_0$ close enough to the origin, the symmetrised metric $\bar{g}_{x_0} = e ^{  2 \bar w_{x_0}} \abs{dx}^ 2$ has finite total $Q$-curvature and non-negative scalar curvature at infinity and at the origin.
\end{lemma} 

\begin{proof} 
If $ g = e ^{2w} \abs{dx}^ 2$ is a metric conformal to the Euclidean metric then by \eqref{eq.Rconf} non-negative scalar curvature $ \R _g \geq 0$ is equivalent to 
\begin{align*}
\Lap w + \big( \tfrac n 2 - 1 \big) \abs{\nabla w}^ 2 \leq 0.
\end{align*} 
Firstly, note that we have 
\begin{align*}
\Lap \bar w _{x_0}(x ) = \Lap \dashint _{ \dB_r (  x_ 0 )}   w ( y) \, d\sigma (y) = \dashint _{ \dB _r ( x_0)}   \Lap w (y) \, d \sigma (y).
\end{align*}
Furthermore, as $\nabla = (\frac{d}{dr}, \frac{1}{r} \nabla_{\Sph^{n-1}})$ in spherical coordinates around $x_0$, we have 
\begin{equation*}
\abs{ \nabla \bar w _{ x_0}}^ 2 = \Big\lvert \frac{d}{dr} \bar w _{ x_0} \Big\rvert^ 2  = \bigg\lvert  \dashint _{ \dB_r(x_0)}   \frac{d}{dr} w \, d \sigma  \bigg\rvert ^ 2 \leq \dashint _{ \dB_r(x_0)}  \Big\lvert \frac{d}{dr} w \Big\rvert^ 2 d \sigma \leq \dashint _{ \dB_r(x_0)}  \abs{\nabla w }^ 2 d \sigma,
\end{equation*}
and therefore 
\begin{equation}\label{eq.Rinteq}
\Lap \bar w_{x_0} + \big( \tfrac n 2 - 1 \big) \abs{ \nabla \bar w_{x_0} }^ 2 \leq \dashint _{ \dB_r(x_0)}  \Big(\Lap w + \big( \tfrac n 2 - 1 \big) \abs{\nabla w}^ 2\Big) d \sigma.
\end{equation}
For $x_0$ and $x$ sufficiently close to the origin, the integrand on the right hand side of \eqref{eq.Rinteq} is non-positive and hence $\R_{\bar g_{x_0}}(x) \geq 0$ if $x_0$ and $x$ are sufficiently close to the origin. Moreover, $\R_{\bar g_{x_0}}(x) \geq 0$ also holds whenever $x$ is sufficiently large. Furthermore $Q_{ \bar g_{x_0}, n }$ is absolutely integrable with respect to $dV _ {\bar  g_{x_0}}$. This follows from
\begin{equation*}
(-\Lap ) ^ { n/ 2 } \bar w_{x_0} = 2 Q_{ \bar g_{x_0}, n } \, e ^ { n \bar w_{x_0} },
\end{equation*}
which implies by Fubini's theorem
\begin{align*}
2\int _{ \RR ^ n } \abs{Q _{ \bar g_{x_0} ,n }} dV_ {  \bar g_{x_0} ,n } &  =\int_{\RR^n} 2 \abs{Q_{ \bar g_{x_0}, n }} \, e ^ { n \bar w_{x_0} } dx =\int _{ \RR ^ n } \big\lvert  ( - \Lap ) ^{ n /2 } \bar w_{x_0}  \big\rvert dx \\
& = \int _{ \RR^ n } \bigg\lvert \, \dashint _{ \dB_r(x_0)}  ( -\Lap ) ^  {n/2 } w ( y ) \, d \sigma (y) \bigg\rvert d x \\
& \leq \int _{ \RR ^ n } \dashint _{ \dB_r(x_0)}  \big\lvert  ( - \Lap ) ^{ n /2 } w(y)  \big\rvert  d \sigma (y) \, dx \\
& = 2 \int _{ \RR^n } | Q_{ g, n } | e^ { n w } d x < \infty.
\end{align*}
Hence $ \bar g_{ x_0 }$ has finite total $Q$-curvature.
\end{proof}

Now, let us consider 
\begin{equation}\label{eq.defv2}
v(x):=\frac{1}{\gamma_n} \int_{\RR^n}\log\Big(\frac{\abs{y}}{\abs{x-y}}\Big)\, Q_{g,n}(y)\,e^{nw(y)}\,dy
\end{equation}
and its symmetrisation
\begin{equation*}
\bar v _{ x_0 } ( x)  = \dashint _{ \dB_r(x_0)}  v (y) \, d \sigma(y), \qquad \text{ where }r = \abs{x -x _0}.
\end{equation*} 
Then, since $ ( -\Lap ) ^{n/2} v = 2Q _{ n ,g } e ^{ nw } = ( -\Lap ) ^{n/2} w$, we also have
\begin{equation*}
( - \Lap  ) ^ { n /2 } \bar v_{ x_0 } (x) = \dashint _{ \dB_r(x_0)}  ( -\Lap ) ^{ n/2} v (y) \, d \sigma(y) = (-\Lap )^{n/2} \bar w _{ x_0 } ( x ).
\end{equation*}
Considering first $x_0=0$, we note that by Lemma \ref{lemma.scalarx0}, $\bar{g} = e^{2 \bar w} \abs{dx}^ 2$ has non-negative scalar curvature at infinity and at the origin and by Lemma \ref{lemma.growth}, $\bar v$ satisfies
\begin{equation*}
\abs{\Lap \bar v(r)} \leq \frac {C}{r^2}, \qquad \abs{\nabla \bar v(r)} \leq \frac{C}{r},
\end{equation*}
for some constant $C \in \RR$. Thus, by an argument as in Lemma \ref{lemma.wisfa}, we obtain
\begin{equation*}
\bar w(x) = \bar v(x) + C_0 + \alpha \log r,
\end{equation*}
for some constants $C_0$ and $\alpha$. Further note that for $x_0 \neq 0$, the Lebesgue differentiation theorem implies that 
\begin{equation*}
\lim_{ r \to 0 } \bar w _{x_0 }  (x)  = w ( x_0),
\end{equation*} 
where $r = \abs{x - x_0}$. Therefore, we obtain
\begin{equation}\label{eq.limit}
\lim_{x \to 0} ( w ( x) - v ( x ) - \alpha \log \abs{x} ) = C _0.
\end{equation}

We can now prove the main result of this section.

\begin{proof} [Proof of Proposition \ref{prop_gen_metric}]
Let $v$ be given by \eqref{eq.defv2} and recall from the above that
\begin{align*}
(-\Lap)^{n/2} ( w - v)= 0 \qquad \text{ on }\RR^n \setminus \{ 0\}.
\end{align*}
Over the origin $(-\Lap)^{n/2} ( w - v)$ must equal the sum of Dirac measures and it's derivatives as they are the unique distributions supported on a point. In view of \eqref{eq.limit}, we conclude that
\begin{equation}\label{eq.Lapn2}
(-\Lap)^{n/2} (w(x) - v(x)  - \alpha \log |x| ) = 0 \qquad \text{ on }\RR^n.
\end{equation}
We first prove the following claim.

\setcounter{claim}{0}
\begin{claim}\label{claimseq5}
$w ( x) - v ( x ) - \alpha \log \abs{x}$ is harmonic on $\RR^n\setminus \{0\}$.
\end{claim}
\begin{proof}
For arbitrary $x_0$, we consider the symmetrisation
\begin{equation*}
\psi_{x_0} (x) = \bar w _{x_0}(x) - \bar v_{x_0}(x) - \alpha \, \overline{ \log} \abs{x} =\dashint _{ \dB_r(x_0)} (w(y) - v(y) - \alpha \log \abs{y})\,  d \sigma(y), 
\end{equation*}
where $r = \abs{x-x_0}$. As the integrand is bounded over the origin by \eqref{eq.limit}, $\psi_{x_0}$ is well defined. Furthermore, $\psi_{ x_0 }$ is rotationally symmetric about $x_0$ and by \eqref{eq.Lapn2}, we have
\begin{align*}
(-\Lap )^{n/2}  \psi_{ x_0 }(x) = 0. 
\end{align*}
Hence, by \eqref{eq.soluformula}, we see  that for $r = \abs{x-x_0}$
\begin{align*}
\psi_{ x_0}(x) = c_1 + c_n\log r + \sum_{k=1}^{\frac{n}{2}-1} c_{2k+1} r^{2k} + \sum_{k=1}^{\frac{n}{2}-1} c_{n-2k} r^{-2k}.
\end{align*}  
We want to conclude that all coefficients apart from $c_1$ vanish. As $\psi_{x_0}$ is bounded at $x_0$, all the even-index coefficients $c_{n-2k}$ (including $c_n$) have to vanish. For the odd-index coefficients $c_{2k+1}$, one can easily reproduce the arguments from Lemma \ref{lemma.wisfa}, noting that by Lemma \ref{lemma.scalarx0}, $\bar{g}_{x_0} = e^{2 \bar w_{x_0}} \abs{dx}^ 2$ has non-negative scalar curvature at infinity and by an argument as in Lemma \ref{lemma.growth}, $\bar v_{x_0} + \alpha \, \overline{ \log} \abs{x}$ satisfies
\begin{equation*}
\abs{\Lap (\bar v_{x_0} + \alpha \, \overline{ \log} \abs{x})} \leq \frac {C}{r^2}, \qquad \abs{\nabla (\bar v_{x_0} + \alpha \, \overline{ \log} \abs{x})} \leq \frac{C}{r},
\end{equation*}
for some constant $C$. This shows that $\psi_{x_0}(x)\equiv c_1$. (Of course, in view of \eqref{eq.limit}, we have $c_1 = C_0$.) Let us point out here that we do not need that $\bar{g}_{x_0} = e^{2 \bar w_{x_0}} \abs{dx}^ 2$ has non-negative scalar curvature at the origin, hence smallness of $\abs{x_0}$ is not needed.\\

Now we conclude that 
\begin{equation*}
\Lap (w(x) - v(x)  - \alpha \log \abs{x} )\big|_{x=x_0} =  \Lap ( \bar w _{ x_0 }(x) - \bar v _{ x_0}(x) - \alpha\,\overline{ \log} \abs{x})\big|_{x=x_0} = \Lap \psi_{x_0}(x) \big|_{x=x_0} = 0.
\end{equation*}
As $x_0$ was arbitrary, this means that $w(x)  - v(x)  - \alpha \log \abs{x}$ is harmonic, proving Claim \ref{claimseq5}. 
\end{proof}

To complete the proof of the proposition, we show that $w(x) - v(x) - \alpha \log \abs{x}$ is in fact a constant. This part of the proof is similar to the four-dimensional case \cite[Section 4]{BuzaNg1}.\\

Since $w(x) - v(x) - \alpha \log \abs{x}$ is harmonic, it follows that 
\begin{align*}
\frac{\partial }{\partial x_k} (w(x) - v(x) - \alpha \log \abs{x})
\end{align*}
is also harmonic. Therefore, using the mean value formula, we get 
\begin{align*}
\bigg\lvert \frac{\partial }{ \partial x_k} (w(x) - v(x) - \alpha \log \abs{x}) \bigg\rvert^ 2 & = \bigg\lvert \, \dashint _{ \dB_r ( x_0) } \frac{\partial }{\partial x_k} (w(x) - v(x) - \alpha \log \abs{x}) \, d \sigma (x)\bigg\rvert^ 2 \\
& \leq  \dashint _{ \dB_r ( x_0) } \Big\lvert \nabla  (w(x) - v(x) - \alpha \log \abs{x}) \Big\rvert^2 d \sigma (x) \\
& \leq C \,  \dashint _{ \dB_r ( x_0) } \Big( \abs{\nabla w(x)}^ 2 + \abs{\nabla v(x)}^ 2 + \abs{\nabla \log\abs{x} }^ 2\Big) d \sigma (x)
\end{align*}
where $ r = \abs{x -x_0}$. Now using the representation formula for $ v (x)$ we get that 
\begin{align*}
&\limsup_{r \to \infty }\, r^2 \dashint _{ \dB_r (x_0) }   \abs{\nabla v}^2 d\sigma (x)\\
&\quad \leq C \limsup _{ r \to \infty }\, r^2 \dashint _{ \dB_r (x_0) } \int_{\RR^n} \frac {\abs{( -\Lap)^{n/2} w(y)}}{\abs{x -y}^2}\, dy\, d \sigma (x)  \int_{\RR^n} \abs{( -\Lap)^{n/2} w (y)} \,dy < \infty. 
\end{align*}
This is due to the fact that
\begin{align*}
\limsup_{ r \to \infty}\, r ^ 2 \dashint _{\dB_r (x_0) } \frac{1}{\abs{x- y}^2 }\, d \sigma(x) \leq C
\end{align*}
by \eqref{eq.estJ}, for any $y\in\RR^n$. Furthermore, using 
\begin{equation*}
\inf_{\RR^n\setminus B_{\tilde{r}_1}(x_0)} \R_g(x) \geq 0
\end{equation*} 
and $ \Lap w(x) = \Lap v(x)  + \alpha\Lap \log \abs{x}$, we obtain from \eqref{eq.Rconf} that
\begin{align*}
\dashint_{ \dB_r ( x_0) }  \abs{\nabla w}^ 2 d \sigma (x) & = \frac {2}{n-2}\; \dashint_{\dB_r ( x_0) } \bigg( \Lap w - \frac{\R_g e^{2w}}{2(n-1)} \bigg) d \sigma(x) \\
&\leq \frac {2}{n-2}\; \dashint_{\dB_r (x_0) } ( \Lap v + \alpha \Lap \log \abs{x} )\, d \sigma(x)\\
& = \dashint_{\dB_r (x_0) }  \bigg( {- \frac{1}{\gamma_n}} \int_{ \RR^n } \frac{(-\Lap)^{n/2} w(y) }{\abs{x-y}^2 } \, dy + \frac{2\alpha}{\abs{x}^2} \bigg) d \sigma(x).
\end{align*}
This shows that 
\begin{align*}
\limsup_{r \to \infty}\;  r^2 \dashint_{\dB_r (x_0)}  \abs{\nabla w}^2  d\sigma(x) < \infty
\end{align*}
and hence 
\begin{align*}
\limsup_{r\to \infty}\; r^2 \Big| \frac{\partial}{\partial x_k} (w(x) - v(x) - \alpha \log\abs{x} ) \Big| < \infty
\end{align*}
which by Liouville's theorem gives that $w(x)  - v(x)  - \alpha \log\abs{x}$ is a constant.
\end{proof}


\section{Multiple ends and cone points}\label{sect.gluing}

In this section we extend the Chern-Gauss-Bonnet formula to cover the case of a domain conformal to $ \mathbb{S}^n$ with several ends and cone points as in Theorem \ref{thm.Mn}. 

\begin{proof}[Proof of Theorem \ref{thm.Mn}]
Let $ \Lambda = \{p_1, \ldots, p_k , q_1, \ldots, q_\ell \}$ be the set of ends and singular points and pick some arbitrary point $N \in \mathbb{S}^n \setminus \Lambda$. Then, we consider the stereographic projection $\pi : \mathbb{S}^n \setminus \{N\} \to \RR^n$ sending $N$ to infinity. We now identify $\Lambda = \{p_1, \ldots, p_k, q_1, \ldots, q _\ell \} $ with its images in $\RR^n$ under this stereographic projection and interpret the metric $g$ on $\Omega$ as a metric on $\RR^n \setminus \{p_1,\ldots,p_k\}$. Hence, there is a function $w$ that is smooth away from $\Lambda \subset \RR^n$ such that 
\begin{equation*}
 g = e^{2w} \abs{dx}^2.
\end{equation*}
Let us then fix a partition of unity $\{\varphi_i(x)\}_{i=0,\ldots k+\ell}$ consisting of smooth functions such that $\varphi_i$ has support in $B_{2R}(p_i)$ and satisfies $\varphi_i \equiv 1$ on $B_R(p_i)$ for $i =1, \ldots, k$, and $\varphi_j$ has support in $B_{2R}(q_j)$ and satisfies $\varphi_j \equiv 1$ on $B_R(q_j)$ for $j = k+1, \ldots k+\ell$. We can make all these balls of radius $2R$ disjoint by choosing $R$ small enough. The function $\varphi_0(x)$ is given by the condition
\begin{align*}
\varphi_0(x) + \varphi_1(x) + \ldots + \varphi_{k+\ell}(x)\equiv 1, \qquad \forall x\in\RR^n.
\end{align*}
Let $ w _i(x) = w (x) \varphi_ i (x)$ and consider the metrics 
\begin{align*}
 g _i = e ^{ 2 w _i } \abs{dx}^2.
\end{align*}  
By assumption, the metrics $g_i$, $i =1, \ldots, k$, are complete with non-negative scalar curvature near the end $p_i$. Moreover, as $w_i \equiv 0$ outside $B_{2R}(p_i)$, we see that $g_i$ has finite total $Q$-curvature and is complete at infinity with zero scalar curvature. Fixing some index $i$, we can assume without loss of generality that $p_i$ is the origin. We consider the symmetrisation  
\begin{align*}
\bar w _{ i} ( x ) = \dashint _{ \dB _r (0) }   w_i (y) \, d \sigma (y), \quad \text{ where } r = \abs{x}.
\end{align*}
Clearly, the arguments from Section \ref{sect.normal} and \ref{sect.general} go through, and we can therefore argue as in Lemma \ref{lemma.2ends} to conclude that
\begin{equation}\label{6.mu}
- \frac{1}{ \gamma_n} \int_{\RR^n} Q_{g_i,n} \, dV_{g_i} = 1+ \nu_i, 
\end{equation} 
where 
\begin{equation*}
\nu_i = \lim_{r\to 0} \frac{\vol_{g_i}(\dB_r(p_i))^{n/(n-1)}}{n\, \sigma_n^{1/(n-1)}\,\vol_{g_i}(B_R(p_i)\setminus B_r(p_i))} =  \lim_{r\to 0} \frac{\vol_g(\dB_r(p_i))^{n/(n-1)}}{n\, \sigma_n^{1/(n-1)}\,\vol_g(B_R(p_i)\setminus B_r(p_i))}
\end{equation*}
and where we used the fact that the Euclidean end at infinity has asymptotic isoperimetric ratio $1$ and that $\varphi_i \equiv 1$ on $B_R(p_i)$.\\

Similarly for $j=k+1,\cdots,k+\ell$, the metrics $g_j$ have a finite area branched point with non-negative scalar curvature near $q_j$ and a complete end with vanishing scalar curvature. By an argument as in Theorem \ref{thm.Rn}, we therefore find 
\begin{align*}
 \chi(\RR^n) - \frac{1}{ \gamma_n } \int_{\RR^n } Q_{g_j, n} \, dV_{g_j}= 1- \mu_j ,
\end{align*} 
or equivalently
\begin{equation}\label{6.nu}
- \frac{1}{ \gamma_n } \int_{\RR^n } Q_{g_j, n} \, dV_{g_j}= - \mu_j ,
\end{equation} 
where 
\begin{equation*}
\mu_j = \lim_{r\to 0} \frac{\vol_{g_j}(\dB_r(q_j))^{n/(n-1)}}{n\, \sigma_n^{1/(n-1)}\,\vol_{g_j}(B_r(q_j))} - 1 = \lim_{r\to 0} \frac{\vol_g(\dB_r(q_j))^{n/(n-1)}}{n\, \sigma_n^{1/(n-1)}\,\vol_g(B_r(q_j))} - 1.
\end{equation*}
Finally, let us also look at $g_0 = e^{2w_0}\abs{dx}^2$. Seen as a metric on $\mathbb{S}^n\setminus \{N\}$, which agrees with $g$ in a neighbourhood of $N$ and therefore can be smoothly extended to all of $\mathbb{S}^n$, we obtain by \eqref{eq.Qn} and \eqref{eq.CGBn}
\begin{equation}\label{6.null}
\chi(\mathbb{S}^n) = \frac{1}{\gamma_n} \int_{\mathbb{S}^n} Q_{g_0,n} \, dV_{g_0} = \frac{1}{2\gamma_n}\int_{\RR^n} (-\Lap)^{n/2} w_0 \, dx.
\end{equation}
Hence, combining the above formulae \eqref{6.mu}--\eqref{6.null} and using that $ w = \sum_{i=0}^ {k+\ell} w _i$, we obtain
\begin{align*}
- \frac{1}{\gamma _n}\int _{\Omega} Q_{g,n}\, dV_g  &= - \frac{1}{ 2\gamma _n} \int_{\RR^n\setminus \Lambda} ( -\triangle)^{n/2} w\, dx\\
&= - \frac{1}{2\gamma_n } \sum_{i=0}^{k+\ell} \, \int _{ \RR^n} (-\triangle ) ^{n/2} w _i \, dx\\
&= - \sum_{i=1}^{k+\ell} \,\frac{1}{\gamma_n}\int_{\RR^n} Q_{g_i,n} \, dV_{g_i} -\chi(\mathbb{S}^n)\\
&= \sum_{i=1}^k (1+\nu_i) - \sum_{j=1}^\ell \mu_j - \chi(\mathbb{S}^n),
\end{align*} 
or equivalently
\begin{equation*}
\chi(\mathbb{S}^n) - k - \frac{1}{\gamma_n}\int _{\Omega} Q_{g,n}\, dV_g = \sum_{i=1}^k \nu_i - \sum_{j=1}^\ell \mu_j.
\end{equation*}
In view of the observation that
\begin{equation*}
\chi(\Omega) = \chi(\mathbb{S}^n\setminus \{p_1,\ldots,p_k\}) = \chi(\mathbb{S}^n) - k,
\end{equation*}
Theorem \ref{thm.Mn} then immediately follows.
\end{proof}

\section*{Acknowledgements}
Parts of this work were carried out during a visit of RB at The University of Queensland in Brisbane as well as two visits of HN at Queen Mary University of London. We would like to thank the two universities for their hospitality. These visits have been financially supported by RB's EPSRC Grant number EP/M011224/1 and HN's AK Head Travelling Scholarship from the Australian Academy of Science.


\makeatletter
\def\@listi{%
  \itemsep=0pt
  \parsep=1pt
  \topsep=1pt}
\makeatother
{\fontsize{10}{11}\selectfont

}
\vspace{10mm}

Reto Buzano (M\"{u}ller)\\
{\sc School of Mathematical Sciences, Queen Mary University of London, London E1 4NS, United Kingdom}\\

Huy The Nguyen\\
{\sc School of Mathematical Sciences, Queen Mary University of London, London E1 4NS, United Kingdom}\\
\end{document}